\documentclass[reqno]{amsart}

\usepackage{amsmath}
\numberwithin{equation}{section}

\usepackage{amssymb}
\usepackage{amsthm}

\theoremstyle{plain}  % default setting: bold heading and italic body
\newtheorem{thm}{Theorem}[section]
\newtheorem{conj}[thm]{Conjecture}
\newtheorem{cor}[thm]{Corollary}

\newtheorem{prop}[thm]{Proposition}

\theoremstyle{definition} % bold heading and normal body
\newtheorem{defn}{Definition}[section]
\newtheorem{examp}[defn]{Example}

\theoremstyle{remark} % italic heading and normal body
\newtheorem{rem}{Remark}[section]

\begin{document}

\title{The Jones polynomial and related properties of some twisted links}

\author{David Emmes}
\address{Unaffiliated}
\email{davidemmes@gmail.com}


\begin{abstract}
Twisted links are obtained from a base link by starting with a
$n$-braid representation, choosing several ($m$) adjacent strands, and
applying one or more twists to the set. Various restrictions may be 
applied, e.g. the twists may be required to be positive or full twists, or the base
braid may be required to have a certain form.

The Jones polynomial of full $m$-twisted links have some interesting properties.
It is known that when sufficiently many full $m$-twists are added that the 
coefficients break up into disjoint blocks which are independent of the 
number of full twists. These blocks are separated by 
constants which alternate in sign. Other features are known.
This paper presents the value of these constants when two strands 
of a three-braid are twisted. 
It also discloses when this pattern emerges for either two or three strand twisting
of a three-braid, along with other properties.

Lorenz links and the equivalent T-links are positively twisted links of
a special form. 
This paper presents the Jones polynomial for such links which have
braid index three.
Some families of braid representations whose closures are identical
links are given.
\end{abstract}

\maketitle

%% main text
\section{Introduction}
The primary purpose of this paper is to display the Jones polynomial for various twisted 
three-braid links and to describe properties of the Jones polynomial for positive three-braids
(Section~\ref{SectionJones3br}). 
These include the family of Lorenz links of
braid index three, $\mathcal{L}_3$. 
In Section~\ref{SectionTlinks} we show exactly when two prototypical braid word representatives 
for $\mathcal{L}_3$ generate the same link
in terms of the braid words themselves. 
A simple unique three-braid representation for links in $\mathcal{L}_3$ is given.
This extends the knot classification result by R.~Bedient, \cite{4}.
For these Lorenz links of braid index (two or) three, the Jones polynomial has the interesting property that
two links are equal exactly when their Jones polynomials are equal.

A second objective of this paper is to present some braid properties of twisted n-braids and show
how these lead to certain symmetries. 
Any pair of n-braid words exhibiting one of these symmetries generates the same link.
These are described in Section~\ref{sectionBraidPropsTlinks} along with their
application to Lorenz links.

Lorenz links were originally investigated in \cite{14}. More recent work by J.~Birman and I.~Kofman, \cite{12},
shows there is a \mbox{one-to-one correspondance} with \mbox{T-links}, which have
a simple braid word description (Section~\ref{ssNotation}).
Section~\ref{SectionTlinks} describes each family of \mbox{T-links} of braid index three.
Thm.~\ref{thmTlinkbr3} shows that the family of three-braid links,
$ \widehat{\sigma_{1}^{x} \sigma_{2}^{y} [1,3]^z} $ with 
$x,y \geq 0$, $z \geq 3$ and $z \equiv 0 \mbox{ mod } 3$
coincides with $\mathcal{L}_3$.
The Jones polynomial for links in $\mathcal{L}_3$ is displayed in 
Cor.~\ref{CorJones3brfulltwists}.

A.~Champanerkar and I.~Kofman have studied the effect on the Jones polynomial
of twisting $m$ strands in a $n$-braid representation of a link, \cite{22}.
The coefficient vector for the Jones polynomial
starts with the coefficient of the term with minimal degree.
Theorem~3.1, \cite{22}, relates that when a set of $m$ strands are given sufficiently many full twists,
the coefficient vector for the Jones polynomial has 
$\lfloor m/2 \rfloor +1$ blocks of fixed length; for a given link the blocks may have different lengths.
The values of the coefficients in each such block is independent of the number of full twists,
and depends only on the base link. 
When $m$ is odd, these blocks are separated by zeros.
When $m$ is even, these blocks are either separated by zeros or by 
a pattern of alternating positive and negative numbers all with the same absolute value.
Furthermore, the length of the inter-block sections increases at a rate proportional to the number of full twists.

These results of A.~Champanerkar and I.~Kofman are given an explicit form for
twisted three-braids and in particular for $\mathcal{L}_3$ (Section~\ref{SectionJones3br}).

\subsection{Background, Notation and Conventions}
\label{ssNotation}

The braid group on $n$ strands is designated $B_n$ and has
$n-1$ standard generators, $\sigma_i$.
The notation $\beta \cong \gamma$, for braid words, $\beta, \gamma$, means that 
$\beta$ and $\gamma$ are conjugate. 
The exponent sum, or writhe, of a braid word, $\beta$, is denoted $w(\beta)$.
The link associated with the 
standard closure of a braid, $\beta$, is denoted $\widehat{\beta}$.
The braid index of a link, $L$ is denoted $b(L)$ and the mirror image of $L$ is denoted $ \overline{L} $.

Braid reflection refers to the replacement of $\sigma_i$ by $\sigma_{n-i}$ in each
instance of a $n$-braid word, $\beta$, for all i, $1 \leq i <n$, yielding a new 
$n$-braid word, $\beta^*$. A very useful result,
\cite{7}, \cite{75}, \cite{10} is that a braid word and its reflection are conjugate, $\beta \cong \beta^*$.

The convenient notation ${[a,b]}= \sigma_a \cdots \sigma_{b-1}$ for $a<b$
and ${[a,b]}= \sigma_{a-1} \cdots \sigma_{b}$ for $a > b$ is borrowed from \cite{12}.
We adopt the convention that ${[a,a]}= 1$.
A typical representation for a \mbox{T-link} is $T((r_1,s_1),\ldots,(r_k,s_k))$, 
with $2 \leq r_1$ and $r_i \leq r_{i+1}$ and $0<s_i$ for all $i$, \cite{12}.
As noted in \cite{12}, we may also assume $s_k \geq 2$.
In this paper, when $r_i < r_{i+1}$ for all $i$,
the parameter $k$ is said to refer to the number of \textit{tiers}. 
A braid representation for $T((r_1,s_1),\ldots,(r_k,s_k))$ is 
${[1,r_1]}^{s_1} \ldots {[1,r_k]}^{s_k}$.

The term \textit{duality} is used to refer to Cor.~4 \cite{12} (general case) or Ex.~3 \cite{12} for
the two-tier case, wherein two specific \mbox{T-links} are shown to represent the same link. These are
$T((r_1,s_1),\ldots,(r_k,s_k))$ and 
$T((\overline{r_1},\overline{s_1}),\ldots,(\overline{r_k},\overline{s_k}) )$, with 
$\overline{r_i} = \sum_{k+1-i}^{k} s_j$, and
$\overline{s_i} = r_{k+1-i} - r_{k-i}$ with the convention  $r_0=0$.
In this context, torus links are just \mbox{T-links} with a single tier.
The elementary torus links are those with two strands, i.e. $r_1=2$, and are
designated $T_{s}=T(2,s)$.

The notation, ${[n]}_{m}$, for positive integers $n, m$ refers to the remainder when $n$ is divided by $m$;
hence $n = m \lfloor n/m \rfloor + {[n]}_{m}$. 
\textit{Parity} refers to ${[n]}_{2}$.
Some convenient abbreviations for this work are to use $G=1+t+t^2$ and $\epsilon_{x} = (-1)^x$.
Based on some prior usage, \cite{105}, when $f \in \mathbb{Z}[t]$,
let $[f]_{a}$ be the coefficient of $t^a$, with $[f]_{\max}$ 
the leading coefficient; $[f]_{\max}=0$ when $f=0$.
An AC (alternating coefficient) polynomial is non-zero and satisfies
$[f]_{j}[f]_{j+1}<0$ for $\min \deg f \leq j < \deg f$.
The Kronecker delta, $\delta_{i,j}$, is one when $i=j$ is true,
and zero otherwise.

The skein relation for the \mbox{Homflypt} polynomial is defined by (\ref{Hskeinrel}), where 
the diagrams, $D_+$, $D_0$, and $D_-$ below refer to the usual diagrams for the
link with positive crossing, null (smoothed) crossing, and negative crossing, respectively.
The standard name for the \mbox{Homflypt} polynomial of an oriented link, $L$, is $P_L$ or $P_L(v,z)$;
for the Jones polynomial $V_L$ or $V_L(t) = P_{L}(t, (t-1)/\sqrt{t})$; 
for the Alexander polynomial, $\Delta_L$ or $\Delta_L(t)= P_{L}(1, (t-1)/\sqrt{t})$; 
for the Conway polynomial, $\nabla_L$ or $\nabla_L(z)= P_{L}(1, z)$.

\noindent
\begin{equation}
P_{D_+}(v,z) =  vz P_{D_0}(v,z) + v^{2} P_{D_-}(v,z).
\label{Hskeinrel}
\end{equation}

We may interpret $A_w=A_w(t)=(t^w+ \epsilon_{w-1})/(t+1)$ as another version of the
Alexander polynomial for $T_w$, and when $w$ is positive, $A_w= \sum_{j=0}^{w -1} \epsilon_j t^{w -1-j}$.
For negative $w$, we have $A_w = \epsilon_{w-1} t^w A_{|w|}$.
Finally, $A_w(1/t)=\epsilon_{w-1} t^{1-w}A_w(t)$.
Many of the results involve manipulations of this simple AC polynomial.
In particular we have for any integers $w,z$ and $x \geq y \geq 0$:

\begin{eqnarray}
A_{w} &=& t^{z} A_{w-z} + \epsilon_{w-z} A_{z}, \label{ARecurwz} \\
%
%A_w &=& t^{w-1}  - A_{w-1}, \label{ARecurwzis1} \\
%
A_{w+z-1} &=& A_{w} A_{z} + t A_{w-1} A_{z-1} \label{AIndexSumAsSumofProd} \\
A_x A_y  
&=& \sum_{j=0}^{y -2} \epsilon_j (j+1) t^{x+y -2-j}
+ (y) \sum_{j=y-1}^{x-1} \epsilon_j  t^{x+y -2-j} \nonumber  \\
&&+ \sum_{j=x}^{x+y -2} \epsilon_j (x+y -1-j) t^{x+y -2-j}. \label{AProductasSum}
\end{eqnarray}

The usual conventions that $\sum_{j=a}^{b} f_j = 
\left( \begin{array} 
{cc} b \\ a
\end{array} \right) = 0$ when $a>b$ are followed.
Also $\prod_{j=a}^{b} f_j = 1$ when $a > b$.
For convenience, the following functions are defined:

\begin{eqnarray*}
L(b) &=& \sum_{j=0}^{b} \epsilon_j (1+b-j) t^{b-j} = \epsilon_b \sum_{j=0}^{b} \epsilon_j (1+j) t^{j}, \\
R(b) &=& \sum_{j=0}^{b} \epsilon_j (j+1) t^{b-j} = \epsilon_b \sum_{j=0}^{b} \epsilon_j (1+b-j) t^{j}. 
\end{eqnarray*}

These may also be expressed as
\begin{eqnarray*}
L(b) &=& \sum_{j=0}^{b} t^{j} A_{b+1-j} = \sum_{j=0}^{b} t^{b-j} A_{j+1}, \\
R(b) &=& \sum_{j=0}^{b} \epsilon_j A_{b+1-j} = \sum_{j=0}^{b} \epsilon_{b-j} A_{j+1}. 
\end{eqnarray*}

These expressions imply 
$L(b) = A_{b+1} + t L(b-1)$ and
$R(b) = A_{b+1} - R(b-1)$.

Hence for $x \geq y \geq 0$ we have the expression (\ref{AProductUsingL}) for $A_x A_y$ 
which partitions the coefficients and terms into three disjoint sets. 
The absolute value of the  coefficients are monotonically increasing for the first $y-1$ terms, 
constant and maximal for the next $x-y+1$ terms, 
and monotonically decreasing for the final $y-1$ terms.
\begin{eqnarray}
A_x A_y  
&=& \epsilon_x  L( y-2 )  - (y) \epsilon_y t^{y-1} A_{x-y+1} +  t^x R(y-2). \label{AProductUsingL}
\end{eqnarray}

The properties just described and reflected in (\ref{AProductUsingL}) are hereditary, as described 
in the next proposition and its corollary. These results are central to the remainder of the paper.

\begin{prop}
Suppose $f, g $ are AC polynomials of degree $d$ with coefficients, $a_j, b_j$ for $t^j$.
Suppose $\Lambda = f + \lambda  t^{d+1} A_{r} + t^{d+r+1} g$ is an AC polynomial with $r>0$, 
and $|a_j| < |a_{j+1}|< |\lambda|$, and $|\lambda| > |b_j| > |b_{j+1}|$ for all $j$.

It follows that $ \Lambda A_y$ is an AC polynomial with the same properties as $\Lambda$ when $1 \leq y \leq r$. 
In particular,
$\Lambda A_y = f_{y}+ \lambda_{y} t^{d+y} A_{r-y+1} +  t^{d+r+1} g_{y}$. 
Here $f_{y}, g_{y}$ are AC polynomials of degree $d+y-1$ 
and $\lambda_{y} = -\epsilon_y (y)\lambda  $.

The absolute values of the coefficients for $f_{y}, g_{y}$ are respectively monotonically increasing, decreasing; and
less than $|\lambda_{y}|$.
Furthermore, the coefficients for $f_{y}$ are dependent on the parity of $r$, but not on the actual value for $r$;
and the coefficients for $g_{y}$ are independent of $r$.
\label{propAlexander}
\end{prop}

\begin{proof}
%\begin{eqnarray*}
%\Lambda A_y &=& f A_y + \lambda t^{d+1} A_{r} A_y + t^{d+r+1} g A_y   \\
%  A_x A_y =\epsilon_x  L( y-2 )  - (y) \epsilon_y t^{y-1} A_{x-y+1} +  t^x R(y-2)
%&=&  f A_y + \lambda t^{d+1} (\epsilon_r  L( y-2 )  - (y) \epsilon_y t^{y-1} A_{r-y+1} +  t^r R(y-2)) + t^{d+r+1} g A_y  \\ 
%
%&=&  f A_y +  \lambda t^{d+1} \epsilon_r  L( y-2 )  - (y)\lambda t^{d+y} \epsilon_y A_{r-y+1} 
%+  \lambda t^{d+r+1} R(y-2) + t^{d+r+1} g A_y  \\
%&& \mbox{So make the following assignments} \\ 
%
By Eq.~\ref{AProductUsingL}, we have 
$f_{y} = f A_y +  \lambda t^{d+1} \epsilon_r  L( y-2 )$ and 
$g_{y} =  \lambda   R(y-2) +  g A_y$. These functions have the properties claimed.
%\end{eqnarray*}
\end{proof}

\begin{cor}
Given a set of $m+1$ positive integers, $\{x_j\}_0^{m}$, with $m \geq 1$ and $x_0 +m-1 \geq s=\sum_{j=1}^{m}  x_j$,
the coefficients of $\prod_{j=0}^{m}  A_{x_j}$ are partitioned as in (\ref{ProdAlexander}):
\begin{eqnarray}
\prod_{j=0}^{m}  A_{x_j} &=& f + \lambda t^{s-m} A_{x_0+m-s}+ t^{x_0} g, \mbox{ with } \lambda = \epsilon_{m+s} \prod_{j=1}^{m}  x_j.
\label{ProdAlexander} %\\
% &=& \epsilon_x  L( y-2 )  - (y) \epsilon_y t^{y-1} A_{x-y+1} +  t^x R(y-2). %\label{AProductUsingL}
\end{eqnarray}
When $x_j =1$ for all $j \geq 1$, we must interpret $f,g$ as zero. Otherwise,
$f,g$ are AC polynomials of degree $s-m-1$ for which the
absolute values of the coefficients are monotonically increasing/decreasing, and
less than $|\lambda|$. 

Furthermore, the coefficients for $f$ are dependent on the parity of $x_0$, but not on the actual value for $x_0$;
and the coefficients for $g$ are independent of $x_0$.
\label{corProdAlexander}
\end{cor}

\section{The Jones and Alexander polynomials for some 3-braid links}
\label{SectionJones3br}
%%%%%%%%%%%%%% Calculating the Jones Polynomial for some simple \mbox{T-links} %%%%%%%%%%%%%%%%%%%%%%%%%%%%
%%%%%%%%%%%%%% Calculating the Jones Polynomial for some simple \mbox{T-links} %%%%%%%%%%%%%%%%%%%%%%%%%%%%

This section briefly introduces some familiar links and their classical skein polynomials and then
describes a decomposition result, Prop.~\ref{PropJones3brfulltwists},
for the Jones polynomial of three braid links that is key to the analysis.
The structure of the Jones polynomial for twisted 3-braids is then presented in
Sections~\ref{SSTwist2} and \ref{SSTwist3}. These sections describe the 
effect of twisting two strands and three strands, respectively, on the Jones polynomial.

The Jones polynomial for the elementary torus link, $T(2,w)$ is
\begin{eqnarray}
V_{T( 2,w ) } &=& -t^{(w-1)/2} ( t^{w+1} + \epsilon_{w} G)/(1+t)  \mbox{, for any } w,
\label{Vfor2br} \\
&=& -t^{(w-1)/2} (\epsilon_{w}+ t^2 A_{w-1}). \label{Vfor2brUsingA}
%
%&=& -t^{(w-1)/2} \{\epsilon_{w}+ \sum_{j=2}^{w} \epsilon_{w-j} t^{j} \} \mbox{, for } w \geq 1. 
%\label{Vfor2brSeries}
\end{eqnarray}

The Alexander polynomial for the elementary torus link, $T(2,w)$ is
\begin{eqnarray}
\Delta_{T( 2,w ) } &=& (t^w+\epsilon_{w-1})/t^{(w-1)/2}(t+1)  \mbox{, for any } w,
\label{Deltafor2br} \\
&=& t^{-(w-1)/2} A_w. \label{Deltafor2brUsingA} 
%
%&=&    \sum_{j=0}^{w-1} \epsilon_j t^{(w-1-2j)/2} \mbox{, for } w \geq 0.
%\label{Deltafor2brSeries}
\end{eqnarray}

The Jones polynomial for any three-braid knot was first given by Prop.~11.10, \cite{41}. 
The expression for any three-braid link  is given by Eq~2.6, \cite{25} and is:
\begin{equation}
V_{\widehat{\beta}}(t) = t^{(w-2)/2} \{t^{w+1} + \epsilon_w G  - G t^{w/2}\Delta_{\widehat{\beta}}(t)\},  \mbox{ with } w=w(\beta).
\label{Vfor3br}
\end{equation}

As the Jones polynomial for three-braid links is determined by the writhe and the
Alexander polynomial, which is sometimes easier to calculate, a few examples are provided. 
If we combine (\ref{Vfor3br}), with the Alexander polynomial formulas for the three braid torus links, 
Section~2.1 \cite{25}, we obtain the following
expressions for the Jones polynomials of torus links on three strands, $T(3,*)$:
\begin{eqnarray}
V_{\widehat{  {[1,3]}^{3a} }}(t) &=&  t^{3a-1} \{1+ t^{2} + 2t^{3a+1} \} , 
\label{Vfor3brTorusLink3comp} \\
%%%%%%%%
V_{\widehat{ {[1,3]}^{3a+b} }}(t) &=& t^{3a+b-1} \{ 1 + t^{2} - t^{3a+b+1} \}   , \mbox{ for } b=1,2.
\label{Vfor3brTorusKnots} 
\end{eqnarray}

Eq.~3.13, \cite{25}, gives a relation (\ref{Homflyrecur}) for the \mbox{Homflypt} polynomial for n-braid links
which provides expressions for the Alexander and Jones polynomials:
\begin{eqnarray}
P_{\widehat{\beta \sigma_i^e}} &=& 
v^{e-1}\,\nabla_{T_e}\,P_{\widehat{\beta \sigma_i}} +
v^{e}\,\nabla_{T_{e-1}}\,P_{\widehat{\beta}}  \,,   \label{Homflyrecur}  \\
\Delta_{\widehat{\beta \sigma_i^e}} &=& \Delta_{T_e}\,\Delta_{\widehat{\beta \sigma_i }} +
\Delta_{T_{e-1}}\,\Delta_{\widehat{\beta}}\,, 
\label{Conwayrecur} \\
V_{\widehat{\beta \sigma_i^e}} &=& 
t^{(e-1)/2} A_e  V_{\widehat{\beta \sigma_i}} +
t^{(e+2)/2} A_{e-1} V_{\widehat{\beta}}  \,.   \label{Jonesrecur} 
\end{eqnarray}

The following result provides a useful way of grouping the terms of the 
Jones polynomial for a special class of three-braids and may be derived
from Eq.~\ref{Jonesrecur}. 
Prop.~\ref{propRank2V3br} implies Cor.~\ref{corRank2Vpos3br} and leads to 
a description of 
the block structure of the Jones polynomial for positive three-braids 
with a highly twisted pair of strands at one site (Prop.~\ref{propVpos3brSpecial}).
\begin{prop}
Given a three-braid word, $\beta = \sigma_1^{a} \sigma_2^{b} \sigma_1^{c} \sigma_2^{d}$, and $w=w(\beta)$,

\begin{eqnarray}
V_{\widehat{\beta}} &=& t^{(w-2)/2} V_{ \widehat{ \beta } }^*, \mbox{ with } 
\label{Vstarfor3br} \\
V_{ \widehat{ \beta } }^* &=& B_{1;\beta}  + t^{a+2} B_{2;\beta} + Q_{\beta}, \nonumber \\
%%%%%%%%%%%
B_{1;\beta} &=& \epsilon_w (1 + t^2 ) +  \epsilon_{a+c}  t^{b+d+1} + \epsilon_{a} t^{c+2} A_{b } A_d  + 2\epsilon_{a+c-1}t^2 A_{b } A_d, \nonumber \\
%%%%%%%%%%%
B_{2;\beta} &=&  \epsilon_{b+d} t^{ c-1} +\epsilon_{d} A_{b}A_{c} + \epsilon_{c} A_{b} A_{d} + \epsilon_{b} A_{c} A_{d}, \nonumber \\
%%%%%%%%%%%%%
Q_{\beta} &=& \epsilon_{c+d} t^2 A_{a}  A_{b} + \epsilon_{b+c}t^2 A_{a} A_{d} + t^3 A_{a } A_{b} A_c A_{d}. \nonumber 
\end{eqnarray}

\label{propRank2V3br}
\end{prop}

Note that $B_{1;\beta}$ is dependent on the parity of $a$, but not on the actual value of $a$, and 
$B_{2;\beta}$ is independent of $a$.
% see Section~\ref{SSProofpropRank2V3br} in older versions.
% R(b) &=& \sum_{j=0}^{b} \epsilon_j (j+1) t^{b-j} = \epsilon_b \sum_{j=0}^{b} \epsilon_j (1+b-j) t^{j}. 
% A_a A_y  &=& \epsilon_a  L( y-2 )  - (y) \epsilon_{y} t^{y-1} A_{a-(y)+1} +  t^a R(y-2)
% A_M A_m  &=& \epsilon_M  L( m-2 )  - (m) \epsilon_{m} t^{m-1} A_{M-(m)+1} +  t^M R(m-2)
Henceforth the implicit definition 
for $V_{ \widehat{ \beta } }^*$ when $\beta \in B_3$ is (\ref{Vstarfor3br}).
$V_{ \widehat{ \beta } }^*$  specifies the 
coefficient vector for positive three-braids.

Three-braid links have the very special property that the skein polynomials under full twists
are determined by the number of full twists and the skein polynomial for the closure of the base braid word.
This decomposition is described by
Prop.~3.5, \cite{25} for the \mbox{Homflypt} polynomial for three-braid links,
and is included here for reference.

\begin{prop}
When $\gamma \in B_3$ and $a >0$ we have:
\begin{eqnarray*}
P_{\widehat {[1,3]^3 \gamma} } &=& v^6 P_{\widehat {\gamma} } + P_{T_{w(\gamma)+5}} - v^6 P_{T_{w(\gamma)+1}} \,, \\
P_{\widehat {[1,3]^{3a} \gamma} } &=& 
v^{6a} P_{\widehat {\gamma} }+ \sum_{j=1}^a v^{6a-6j} P_{T_{w(\gamma)+6j-1}} - \sum_{j=1}^a v^{6+6a-6j} P_{T_{w(\gamma)+6j-5}}\,, \\
P_{\widehat {[1,3]^{-3a} \gamma} } &=& 
v^{-6a} P_{\widehat {\gamma} }+ \sum_{j=1}^a v^{6j-6a} P_{T_{w(\gamma)-6j+1}} - \sum_{j=1}^a v^{6j-6a-6} P_{T_{w(\gamma)-6j+5}}\,.
\end{eqnarray*}
\label{Homflypt3brfulltwists}
\end{prop}

When Prop.~\ref{Homflypt3brfulltwists} is applied to the Jones polynomial, we obtain
\begin{prop}
When $\gamma \in B_3$, set $\beta = [1,3]^{3a} \gamma$. For any $a$, we have:
\begin{eqnarray}
V_{\widehat {[1,3]^{3a} \gamma} }
&=& t^{(w(\beta)-2)/2} \{ \epsilon_{w(\gamma)}   ( 1+t^{2} )  + t^{3a} B_2(t,\gamma)  \}, \mbox{ with } 
\label{EqJones3brfulltwists}  \\
B_2(t,\gamma) &=& t^{(2-w(\gamma))/2} V_{\widehat {\gamma} } + \epsilon_{w(\gamma)+1}  ( 1+t^{2})  \,.
\label{EqJones3brfulltwistsBlock2Gen}
\end{eqnarray}

\label{PropJones3brfulltwists}
\end{prop}

Note that when $\gamma$ is (conjugate to) a positive three-braid word the minimum degree of $V_{\widehat {\gamma} }$
is $(w(\gamma)-2)/2$, so that the expression $t^{(2-w(\gamma))/2} V_{\widehat {\gamma} }$ in
(\ref{EqJones3brfulltwistsBlock2Gen}) is an ordinary polynomial with nonzero constant, i.e. $V_{\widehat {\gamma} }^{*}$.
If we focus on the effect of full twists of all three strands, 
the two blocks in the coefficient vector are already separated
when $\gamma$ is (conjugate to) a positive braid word and $a > 0$. As $a$ is increased, the two blocks separate
by three for each new full twist that is added.

By \cite{27}, \cite{110}, when $\gamma$ is (conjugate to) a positive three-braid word and $\widehat{\gamma}$ is not split,
the first three values in the coefficient vector are 
$\epsilon_{w(\gamma)},0,\epsilon_{w(\gamma)} p(\widehat{\gamma})$, where $p(\widehat{\gamma})$ is
the number of prime factors of $\widehat{\gamma}$. 
Under these conditions when $\widehat{\gamma}$ is also prime, $B_2(t,\gamma)$ is divisible by $t^3$;
otherwise $B_2(t,\gamma)$ is divisible by $t^2$.

\begin{examp}
By \cite{21}, two simple knots are $8_{19} = \widehat{\sigma_1^3 \sigma_2 \sigma_1^3 \sigma_2}$ and 
$10_{124} = \widehat{\sigma_1^5 \sigma_2 \sigma_1^3 \sigma_2}$.
As $8_{19} = \widehat{[1,3]^3 \sigma_2 \sigma_1}$
and $10_{124} = \widehat{[1,3]^3  \sigma_1^3 \sigma_2}$, Prop.~\ref{PropJones3brfulltwists} gives us
\begin{equation*}
V_{8_{19} } =
t^{3} \{  ( 1+t^{2} )  + t^{3} B_2(t,\sigma_2 \sigma_1)  \}, \mbox{with }
B_2(t,\sigma_2 \sigma_1) = 1 - ( 1+t^{2}) = - t^{2}\,.   \\
\end{equation*}
\begin{equation*}
V_{10_{124} } =
t^{4} \{     ( 1+t^{2} )  + t^{3} B_2(t, \sigma_1^3 \sigma_2)  \}, \mbox{with }
B_2(t, \sigma_1^3 \sigma_2) = t^{-1} V_{T_3 } - ( 1+t^{2}) = -t^3\,.
\end{equation*}

\end{examp}

As any three-braid word, $\gamma$, has a (non-unique) representation as $[1,3]^{3b} \eta$, where $\eta$ 
is itself a positive three-braid word, the prior comments extend to all three-braid words.
When $b$ is chosen to be maximal, the decomposition in 
(\ref{EqJones3brMaxfulltwists}, \ref{EqJones3brfulltwistsBlock2})
shows that even when $b \leq 0$, the two blocks are separated when $a > |\,b\,|$\,.
Note that $w(\gamma) \equiv w(\eta) \mbox{ mod } 2$. 
Even for submaximal $b$, we have

\begin{eqnarray}
V_{\widehat {[1,3]^{3a} \gamma} }
&=& t^{(w(\beta)-2)/2} \{ \epsilon_{w(\gamma)}   ( 1+t^{2} )  + t^{3a+3b} B_2(t,\eta)  \}, \mbox{ with } 
\label{EqJones3brMaxfulltwists} \\
B_2(t,\eta) &=& t^{(2-w(\eta))/2} V_{\widehat {\eta} } + \epsilon_{w(\eta)+1}  ( 1+t^{2})  \in \mathbb{Z}{[t]} \,.
\label{EqJones3brfulltwistsBlock2}
\end{eqnarray}

\begin{examp}
One of the hyperbolic knots from \cite{23} and \cite{21} is
$\mathbf{k}7_{80} = 10_{161}$ with representation 
$ \sigma_1^{3} \sigma_2^{1} \sigma_1^{-1} \sigma_2^{1} \sigma_1^{2} \sigma_2^{2}$.
The mirror image has a braid expression as  $[1,3]^{-9} \sigma_1^{9}\sigma_2$, so

\begin{eqnarray}
V_{ \overline{10_{161}} }
&=& t^{-5} \{ ( 1+t^{2} )  + t^{-9} B_2(t,\sigma_1^{9}\sigma_2)  \}, \mbox{ with } 
\label{VforMirrorof10161} \\
B_2(t,\sigma_1^{9}\sigma_2) &=& t^{-4} V_{T_9 } - ( 1+t^{2}) =-t^3 A_7 \,.
\end{eqnarray}
Here we observe the overlap of the terms $1+t^{2}$ and $t^{-9} B_2(t,\sigma_1^{9}\sigma_2)$ in
(\ref{VforMirrorof10161}). 

Since $V_{ 10_{161}  } (t) = V_{ \overline{10_{161} } }(1/t)$ we find 
$V_{ 10_{161} } = t^{3} ( 1 - t^{3} A_6)$.
\end{examp}

Again focusing on the effect of increasing $a$, 
for any three-braid word there are only finitely many values of $a$ for which the two blocks overlap,
and this happens exactly when $3a+\min \deg B_2(t,\gamma) \leq 2$ and $3a+\max \deg B_2(t,\gamma) \geq 0$,
i.e. $-\max \deg B_2(t,\gamma) \leq 3a \leq 2- \min \deg B_2(t,\gamma)$.
Eq.~\ref{EqJones3brfulltwists} displays one block in the coefficient vector
as the triplet $\epsilon_{w(\gamma)}, 0, \epsilon_{w(\gamma)}$.
Eq.~\ref{EqJones3brfulltwistsBlock2} displays the second block as 
$B_2(t,\eta) = t^{-3b} B_2(t,\gamma) \in \mathbb{Z}{[t]}$, with the possibility that several of the 
lowest order terms may be zero.

The significance of the prior expressions
relative to the effect of twisting two strands is that they show it suffices to look at
the effect of such twisting applied to a positive braid word. 
Indeed, given a base braid word $\gamma = \prod_{i=1}^{r} \sigma_1^{x_i} \sigma_2^{y_i}$ 
and a twist site within $\gamma$, we 
may use conjugacy and/or braid reflection to ensure that the site is the first 
generator in the expression for $\gamma$. 
Now write $ \sigma_2^{y_1} \prod_{i=2}^{r} \sigma_1^{x_i} \sigma_2^{y_i}$ as 
$\eta {[1,3]}^{3b}$ as above, and allow $x_1$ to increase by increments of two, where
we may assume that $x_1 >0$.
By Prop.~\ref{PropJones3brfulltwists}, once $x_1$ is sufficiently large, 
the second block, $t^{3b} B_2(t,\sigma_1^{x_1}\eta)$, has 
minimal degree above two, so that the two blocks in the coefficient vector
are clearly recognizable when any further full twists of the two strands are added.

\begin{defn}
If $\beta = \prod_{j=1}^{r} \sigma_1^{e_{2j-1}} \sigma_2^{e_{2j}} $ with $r \geq 1$ and all $e_{k} \neq 0$,
call $\sigma_i^{e_k}$ a syllable. A syllable is trivial when the exponent is one. 
A trivial syllable is isolated if both adjacent syllables, viewed cyclically, are non-trivial.
Denote the rank of $\beta$ as $\rho(\beta) = r$ and assign a rank of zero for the identity of $B_3$,
and a rank of one for $\sigma_i^{a}$ for $a \neq 0$.
\end{defn}

\begin{prop}
Given a positive three-braid word, $\beta = \prod_{j=1}^{r} \sigma_1^{a_j} \sigma_2^{b_j} $, with $r \geq 1$,
and exponents $a_j, b_j > 0$,
and $w=w(\beta)$, we have
\begin{eqnarray*}
V_{\widehat{\beta}}(t) &=& t^{(w-2)/2} V_{ \widehat{ \beta } }^*(t), 
\mbox{ with } V_{ \widehat{ \beta } }^*(t) \in \mathbb{Z}[t] \mbox{ and }
\deg V_{ \widehat{ \beta } }^*(t) \leq w \mbox{ and } \nonumber \\
V_{ \widehat{ \beta } }^*(t) &=& \epsilon_{w}(1+t^{2}) + t^2 V_{ \widehat{ \beta } }^{**}(t),
\mbox{ with } V_{ \widehat{ \beta } }^{**}(t) \in \mathbb{Z}[t].
\end{eqnarray*}
Furthermore,  $\epsilon_{w} t^{w} V_{ \widehat{ \beta } }^{**}(1/t) = t^2 V_{ \widehat{ \beta } }^{**}(t)$,
or equivalently $[V_{ \widehat{ \beta } }^{**}]_{j} = \epsilon_{w} [V_{ \widehat{ \beta } }^{**}]_{w-j-2}$.

When $r=1$, we have $V_{ \widehat{ \beta } }^{**}$ is an AC polynomial, or zero, with the following properties:
\begin{enumerate}
\item $V_{ \widehat{ \beta } }^{**}(t) =0$ exactly when $\{a_1, b_1 \} = \{ 1,2 \}$,
\item $\deg V_{ \widehat{ \beta } }^{**}(t) = w-3$ exactly when $\min(a_1,b_1)=1$ and $\max(a_1,b_1) \geq 3$;
in which case $V_{ \widehat{ \beta } }^{**}(t) = - t A_{w-3} $,
\item $\deg V_{ \widehat{ \beta } }^{**}(t) = w-2$ exactly when $a_1=b_1=1$, or $a_1, b_1 \geq 2$; 
here $V_{ \widehat{ \beta } }^{**}(t) = -1$ or 
$V_{ \widehat{ \beta } }^{**}(t) = \epsilon_w+ \epsilon_{b_1} t  A_{a_1-2} + \epsilon_{a_1} t A_{b_1-2} + t^2 A_{a_1-1} A_{b_1-1}$
respectively,
\item the sign of $[V_{ \widehat{ \beta } }^{**}]_{j}$ is $\epsilon_{j+w}$ when $w \geq 4$.
\end{enumerate}

When $r\geq 2$, we have $\deg V_{ \widehat{ \beta } }^{**}(t) \leq w-3$.

\label{JonesPos3brProps}
\end{prop}
\begin{proof}
When $r=1$, write $\beta = \sigma_1^{a} \sigma_2^{b}$. We may assume $a\geq b$.
When $b=1$, the result is clear, as 
$V_{\widehat{\beta}}(t) = -t^{(a-1)/2} (\epsilon_{a}+ t^2 A_{a-1})$, which is
$%t^{(w-2)/2} (-\epsilon_{w-1} - t^2 A_{w-2}) = 
t^{(w-2)/2} (\epsilon_{w} - t^2 A_{w-2})$.
If $a=1$, take $V_{ \widehat{ \beta } }^{**}(t) = -1$.
If $a \geq 2$, % note $V_{ \widehat{ \beta } }^*(t) =\epsilon_{w} - t^2 A_{w-2} = \epsilon_{w}(1+t^2) - t^2 (t A_{w-3})$, so
let $V_{ \widehat{ \beta } }^{**}(t) = - t A_{w-3} $.
The properties claimed are easily verified.

By (\ref{Vfor2brUsingA}), we may take 
$V_{ \widehat{ \beta } }^{**}(t) = \epsilon_w+ \epsilon_{b} t  A_{a-2} + \epsilon_{a} t A_{b-2} + t^2 A_{a-1} A_{b-1}$
when $a,b \geq 2$. This is a sum of AC polynomials each with the sign of $t^j$ as claimed.
The other properties claimed for $r=1$ are easily verified.

For $r \geq 2$, apply (\ref{Jonesrecur}) to obtain the following relations for 
$V_{ \widehat{ \gamma \sigma_1^e \sigma_2^f } }^{**}(t)$
which utilize terms that depend on braids of lower rank and so allows an induction proof.
Here $\beta = \gamma \sigma_1^e \sigma_2^f$ with
$\gamma = \sigma_1^a \eta \sigma_2^d$ and $\eta$ may be the identity.
\begin{eqnarray}
V_{ \widehat{ \gamma \sigma_1^{e} \sigma_2^{1} } }^{**} 
&=& A_e V_{\widehat{ \sigma_1^{a+d} \eta \sigma_1 \sigma_2^{1}}  }^{**} 
 + t A_{e-1}   V_{\widehat{ \sigma_1^{a} \eta \sigma_2^{d+1}  }}^{**}\,, 
\label{proofpropJonesPos3brPropsfis1}  \\
%%%%%%%%
V_{ \widehat{ \gamma \sigma_1^{e} \sigma_2^{f} } }^{**} 
&=& A_f V_{\widehat{ \sigma_1^a \eta \sigma_2^d \sigma_1^{e} \sigma_2^{1} }  }^{**} 
 + t A_{f-1}   V_{\widehat{ \sigma_1^{a+e} \eta \sigma_2^{d}  }}^{**}\,. 
\label{proofpropJonesPos3brPropsfgeq2}  %\\
%%%
%&=&
%A_e\, A_f    V_{\widehat{ \sigma_1^{a+d} \eta \sigma_1 \sigma_2^{1}}  }^{**}  
% + t A_e\, A_{f-1}    V_{\widehat{ \sigma_1^{a+1} \eta \sigma_2 ^d }}^{**}  \nonumber \\
%%
%&&  +  t A_{e-1}\,A_f   V_{\widehat{ \sigma_1^{a} \eta \sigma_2^{d+1}  }}^{**}   
%  + t^{2} A_{e-1}\, A_{f-1}   V_{\widehat{ \sigma_1^{a} \eta \sigma_2^{d}  }}^{**}.
%\label{proofVRank2Pos3br}
\end{eqnarray}

To verify
$\epsilon_{w} t^{w} V_{ \widehat{ \beta } }^{**}(1/t) = t^2 V_{ \widehat{ \beta } }^{**}(t)$ is straightforward
using (\ref{proofpropJonesPos3brPropsfis1}), (\ref{proofpropJonesPos3brPropsfgeq2}), 
as is the degree bound $\deg V_{ \widehat{ \beta } }^{**}(t) \leq w-3$.

\end{proof}

$V_{ \widehat{ \beta } }^{**}$ appears to be an
AC polynomial for positive three-braids of rank two or more; 
the sign of the coefficients is a simple 
function of $w(\beta)$, $j$ and $r$ (Conj.~\ref{JonesPos3brConjs}). 
This may be verified directly for $\rho(\beta)=2,3$, but 
these cases do not readily suggest a general pattern or proof. 
The following proposition identifies two cases when a positive three-braid of rank two or more 
is conjugate to a braid of the form $[1,3]^{3a} \gamma$, 
with $a \geq 0$ and $\gamma$ a non-negative three-braid of lower rank.
Assuming Conjecture~\ref{JonesPos3brConjs} is true, Prop.~\ref{PropJones3brfulltwists} shows these are the only such cases.

\begin{prop}
Assume $\beta$ is a three-braid word, 
$\prod_{j=1}^{r} \sigma_1^{a_j} \sigma_2^{b_j} $ with $r \geq 2$ and all $a_j, b_j \geq 1$.
In either case below, we have
$\beta \cong [1,3]^{3a} \gamma$ with $a \geq 0$ and $\rho(\gamma) < r$.

\begin{enumerate}
\item $\beta$ has two or more trivial syllables, or
\label{prop3br2TrivialSyllables}
\item $\beta$ has a trivial syllable whose minimum adjacent syllable length is two;
here $\sigma_1^{a_1}$ and $\sigma_2^{b_r}$ are considered to be adjacent syllables.
\label{prop3brTrivialSyllable212}
\end{enumerate}

In case~\ref{prop3brTrivialSyllable212} or when two isolated trivial syllables exist, we have $a \geq 1$.

We may choose $\gamma =\prod_{j=1}^{s} \sigma_1^{c_j} \sigma_2^{d_j}$ with $s < r$
and $c_j, d_j \geq 2$ for all $j \geq 2$. 

For $s \geq 2$, we have $c_1,d_1 \geq 1$ 
and we may choose $\gamma$ to either have no trivial syllables, or one trivial syllable  
whose two adjacent syllables each have a minimum length of three.

For $s=1$, we have $c_1, d_1 \geq 0$, and for
$s = 0$, we have $\gamma = 1$.

\label{prop3brTrivialSyllable}
\end{prop}

%Note also that $\sigma_1 \sigma_2^3 \sigma_1 \sigma_2 \cong \sigma_1^5 \sigma_2 \not \cong [1,3]^3$.

\begin{proof}
If two trivial syllables of $\beta$ are adjacent (in a cyclical sense),
we may assume they correspond to $a_r = b_r=1$. 
When $r=2$, let $c_1=a_1 + b_1+1$ and $d_1=1$ with $s=1$ and $a=0$.
When $r \geq 3$ let $c_1=a_1 + b_{r-1}$, $c_{r-1} = a_{r-1} + 1$, $d_1=b_1$ and $d_{r-1} =1$ with 
$c_j=a_j$ and $d_j= b_j$ for $2 \leq j \leq r-2$ with $s=r-1$ and $a=0$.
As $\gamma$ has smaller rank, induction applies.

Suppose there is a pattern $\sigma_1^{a_j} \sigma_2 \sigma_1^{a_{j+1}}$ in $\beta$ or its
braid reflection, and that $\min(a_j,a_{j+1}) = 2$.
With $r=2$ and $\beta = \sigma_1^{a_1} \sigma_2 \sigma_1^{2} \sigma_2^{b_{2}}$
we have $\beta = [1,3]^3 \sigma_1^{a_1-2} \sigma_2^{b_{2}-1}$.
A similar result applies when $r \geq 3$.
As $\gamma$ has smaller rank, induction applies.

The only remaining case is when two or more trivial syllables exist and 
each neighboring syllable has length at least three. 
If there is a subword such as $\sigma_1^{a_1} \sigma_2 \sigma_1^{a_2} \sigma_2 \eta$   
and $\eta$ is either vacuous or begins with $\sigma_1$, this is just
%$\sigma_1^{a_1 - 1} \sigma_1 \sigma_2 \sigma_1^{a_2} \sigma_2 \eta$
$\sigma_1^{a_1 - 1} \sigma_2^{a_2} \sigma_1 \sigma_2 \sigma_2 \eta$
which has already been handled. This leads to an easy induction using
the distance between two trivial syllables. 
We may assume $r \geq 3$ as the $r=2$ case was just addressed. 
Assume there is a subword such as $\sigma_1^{a_1} \sigma_2^1 \sigma_1^{a_2} \eta \sigma_k^1 \xi$,
where $k$ may be either $1$ or $2$, and $\eta$ begins with $\sigma_2^{b_2}$, does not end in $\sigma_k$, 
and all twist exponents in $\eta$ are at least two.
This subword is just
%$\sigma_1^{a_1-1}    \sigma_1  \sigma_2^1 \sigma_1^{a_2}    \eta \sigma_k^1 \xi$
$\sigma_1^{a_1-1} \sigma_2^{a_2} \sigma_1^1 \sigma_2      \eta \sigma_k^1 \xi$.
The case $b_2=1$ has already been handled, otherwise we have created a braid word 
with the same properties but with reduced distance between the trivial syllables. 
\end{proof}

\begin{defn}
When a positive three braid of rank two or more
%has the form $\gamma = \prod_{j=1}^{s} \sigma_1^{e_{2j-1}} \sigma_2^{e_{2j}}$
%with $s \geq 2$ and 
either has no trivial syllables, or a single trivial syllable whose 
two adjacent syllables each have a minimum length of three, 
the three-braid is called \textit{condensed}.
\end{defn}

\begin{conj}
Given a condensed three-braid word, $\beta$, with $w=w(\beta)$, we have
$V_{ \widehat{ \beta } }^*(t) = \epsilon_{w}(1+t^{2}) + t^2 V_{ \widehat{ \beta } }^{**}(t)$ and

\begin{enumerate}
\item the sign of $[V_{ \widehat{ \beta } }^{**}]_{j}$ is $\epsilon_{r+1+j+w}$, 
so $V_{ \widehat{ \beta } }^{**}$ is an AC polynomial,
\item when $\beta$ has no trivial syllables, $\deg V_{ \widehat{ \beta } }^{**}(t) = w-r-1$,
and $[V_{ \widehat{ \beta } }^{**}]_{\max}=1$,
\item when $\beta$ has one trivial syllable, $\deg V_{ \widehat{ \beta } }^{**}(t) = w-r-2$, 
and $[V_{ \widehat{ \beta } }^{**}]_{\max}=-1$,
\item $[V_{ \widehat{ \beta } }^{**}]_{j} \neq 0$ 
exactly when $j \in [w-2-\deg V_{ \widehat{ \beta } }^{**}(t),\,\deg V_{ \widehat{ \beta } }^{**}(t)]$.
\end{enumerate}

\label{JonesPos3brConjs}
\end{conj}

\begin{examp} 
If Conj.~\ref{JonesPos3brConjs} is true, Prop.~\ref{PropJones3brfulltwists} implies 
that $\deg V_{\widehat{\beta}}^{**}$ is lower than the values given by Conj.~\ref{JonesPos3brConjs}
when $\beta$, of rank at least two, has only isolated trivial syllables but is not condensed. 
The condensed braid, $\sigma_1^{3} \sigma_2^{2} \sigma_1^{2} \sigma_2^{3}$ has 
$V_{10_{152} }^{**} = V_{ \widehat{ \sigma_1^{3} \sigma_2^{2} \sigma_1^{2} \sigma_2^{3} } }^{**} = 
t - 2t^2 + 2t^3 - 3t^4 + 2t^5 - 2t^6 + t^7$ and satisfies Conj.~\ref{JonesPos3brConjs}.
\begin{eqnarray*} 
% 12_242 is (7,2,2,1)
V_{12n_{242} }^{**} = V_{ \widehat{ \sigma_1^{7} \sigma_2^{2} \mathbf{\sigma_1^{2}} \sigma_2^{1} } }^{**} &=& \epsilon_{1} t^4 A_3 
\mbox{ has degree }  6, \beta \cong [1,3]^3 \sigma_1^5 \sigma_2,   \\
% 12n_472 is (5,2,4,1)
V_{12n_{472} }^{**} =V_{ \widehat{ \sigma_1^{5} \sigma_2^{2} \sigma_1^{4} \sigma_2^{1} } }^{**} &=&  \epsilon_{1} t^2 A_{3} A_5 
\mbox{ has degree }  8, \\
%%% 12n_574 is (3,2,6,1)
V_{12n_{574} }^{**} =V_{ \widehat{ \sigma_1^{3} \sigma_2^{2} \sigma_1^{6} \sigma_2^{1} } }^{**} &=&
\epsilon_{1} t^2 [ 1-t+2t^2-3t^3+2t^4-t^5+t^6 ]  \mbox{ has degree }  8, \\ 
%%% 12n_725 is (5,4,2,1)
V_{12n_{725} }^{**} = V_{ \widehat{ \sigma_1^{5} \sigma_2^{4} \mathbf{\sigma_1^{2}} \sigma_2^{1} } }^{**} &=&  
t^3 [ 1 - 2t +t^2  -2t^3  + t^4  ] \mbox{ has degree }  7,
\beta \cong [1,3]^3 \sigma_1^3 \sigma_2^3.
\end{eqnarray*}
\end{examp}

In the following sections it is desirable to group the terms of the Jones polynomial 
for three-braid links into three components, i.e. the first block, the inter-block gap,
and the second block. Toward this end the following proposition is helpful in many 
calculations. Typically $g_1$ and $b_2$ are chosen as the minimum degrees for the first term of
the inter-block gap and the second block, respectively.

\begin{prop}
Suppose that $k \geq 0 $ and $s, g_1, b_2 >0$ are given with $k \leq g_1 \leq b_2$. We have the following disjoint
partition of $t^k A_s$, with $A_{z_i} \in \mathbb{Z}[t]$:
\begin{eqnarray*}
t^k A_s &=& \lambda_1 t^k A_{z_1} + \lambda_2 t^{g_1} A_{z_2}+ t^{b_2} A_{z_3}, \mbox{ with } \\
z_1 &=& g_1-k \mbox{ when } g_1 \leq k+s , \mbox{ and } \lambda_1= \epsilon_{g_1+k+s}, \\ 
&=& s \mbox{ when } k+s  \leq g_1, \mbox{ and } \lambda_1= 1, \\ 
z_2 &=& b_2-g_1 \mbox{ when } b_2  \leq k+s , \mbox{ and } \lambda_2=  \epsilon_{b_2+k+s}, \\ 
&=& k+s -g_1  \mbox{ when } g_1  \leq k+s \leq b_2, \mbox{ and } \lambda_2= 1,  \\
&=& 0 \mbox{ when } k+s  \leq g_1, \mbox{ and } \lambda_2= 1, \\ 
z_3 &=& k+s-b_2 \mbox{ when } k+s  \geq b_2, \\ 
&=& 0 \mbox{ when } k+s  \leq b_2. 
\end{eqnarray*}

\label{propPartitionShiftA}
\end{prop}

The next section describes the properties of each of the components of 
the coefficient vector for the closure of a positive braid word: 
the length of the first block, the length of the inter-block gap and its coefficients,
and the maximal length of the second block. 
The number of twists required to ensure that the two blocks are disjoint
is also provided.

%%%%%%%%%%%%%%%%%%%%%%% Twisting two strands of a three-braid link %%%%%%%%%%%%%%%%%%%%%
%%%%%%%%%%%%%%%%%%%%%%% Twisting two strands of a three-braid link %%%%%%%%%%%%%%%%%%%%%
%%%%%%%%%%%%%%%%%%%%%%% Twisting two strands of a three-braid link %%%%%%%%%%%%%%%%%%%%%

\subsection{Twisting two strands of a three-braid link}
\label{SSTwist2}

If we apply Prop.~\ref{propRank2V3br} to positive three-braid words with a dominant twist exponent we 
obtain an expression, Cor.~\ref{corRank2Vpos3br}, for the Jones polynomial that displays its block characteristics.
The main steps in the proof use 
Eq.~\ref{ProdAlexander} applied to the term $A_a A_b A_c A_d$ in $Q_\beta$. 
Prop.~\ref{propPartitionShiftA} and (\ref{AProductUsingL}) are the other helpful tools needed.

\begin{cor}
Given a positive three-braid word, $\beta = \sigma_1^{a} \sigma_2^{b} \sigma_1^{c} \sigma_2^{d}$, 
set $w=w(\beta)$ and $w^* = w-a$,
with $a \geq w^*$.
We have the following partitioning of coefficients:

\begin{eqnarray}
V_{ \widehat{ \beta } }^* &=& B_{1}(\beta) + t^{w^*+1} G(\beta) + t^{a+2}B_{2}(\beta),  \mbox{ with }
\label{Vfor3brRank2} \\
%%%%%%
B_{1}(\beta) &=& \epsilon_w (1 + t^2 ) 
+  \epsilon_{a+c}  t^{b+d+1} + \epsilon_{a} t^{c+2} A_{b } A_d  + 2\epsilon_{a+c-1}t^2 A_{b } A_d  \nonumber \\
&& + \epsilon_{a+c+d} t^2 L( b-2 ) + (b) \epsilon_{a}  t^{b+1}  A_{c+d} \nonumber \\
&&+  \epsilon_{a+b+c} t^2 L( d-2 ) + (d) \epsilon_{a}  t^{d+1}  A_{b+c}  \nonumber \\
&&+  t^3 f_{\beta} + \epsilon_{a+1} b\,c\,d\,t^{w^*} , \\
%%%
G(\beta) &=& \epsilon_{w^*}(b\,c\,d - b -d) A_{a+1-w^*},  \\
%%%
B_{2}(\beta) &=& \epsilon_{b+d} t^{ c-1} +\epsilon_{d} A_{b}A_{c} + \epsilon_{c} A_{b} A_{d} + \epsilon_{b} A_{c} A_{d} \nonumber \\
&& + \epsilon_{c+d}  R(b-2) + \epsilon_{b+c} R(d-2) + \epsilon_{1+w^*} b\,c\,d  + t g_{\beta}.
\end{eqnarray}

Here $f_{\beta}, g_{\beta}$ are the $f,g$ in Cor.~\ref{corProdAlexander} for the product $A_a A_b A_c A_d$.
We have $\deg f_{\beta} = \deg g_{\beta} = w^*-4$. 
As in Cor.~\ref{corProdAlexander}, we must interpret $f_{\beta} = g_{\beta} = 0$ 
when $b=c=d=1$; in this case $\widehat{\beta} = T_{a+2}$. 
%with $B_{1}(\beta) = \epsilon_w (1 + t^2 -t^3)$, $G(\beta) = A_{a-2}$, and $B_{2}(\beta)=-1$.

The maximum degree for $B_{1}(\beta)$ is $w^*$, while that for $B_{2}(\beta)$ is $w^*-3$,
i.e. $\deg V^* \leq w-1$. The latter upper bound is only achieved when $\min(b,d)=c=1$ $(\widehat{\beta} = T_{w-1})$, or
when $\min(b,c,d) \geq 2$.

\label{corRank2Vpos3br}
\end{cor}
% see Section~\ref{SSProofcorRank2Vpos3br} in older versions.

It is apparent from these expressions that $\deg B_{1}(\beta) \leq b+c+d$, whereas $\deg t^{b+c+d+1} G(\beta) \leq a+1$, so
the coefficients are partitioned. As foreshadowed by Prop.~\ref{propRank2V3br}, 
$B_{1}(\beta)$ is dependent on the parity of $a$, but not on the actual value of $a$, and 
$B_{2}(\beta)$ is independent of $a$. 
The coefficients of the inter-block gap, $G(\beta)$, are independent of $a$, and the width 
depends linearly on $a$.
Note that $G(\beta)=0$ exactly when $(b,c,d)$ is $(1,2,1)$ or $(2,1,2)$; in these cases $\beta = \sigma_1^{a-2}[1,3]^3$ and
$\beta = \sigma_1^{a-1}[1,3]^3$, respectively.

The coefficient of $t^{b+c+d}$ in (\ref{Vfor3brRank2}) is 
$\delta_{c,1} \epsilon_{a+c} + \epsilon_{a}(1+b+d - b\,c\,d)$, which differs from the
plateau of the inter-block gap, except when $c=1$. Similarly,
the coefficient of $t^{a+2}$ is % $\delta_{c,1} \epsilon_{b+d} + \epsilon_{b+c+d} (3 + b-1 + d-1- b\,c\,d)$, i.e.
$\delta_{c,1} \epsilon_{b+d} + \epsilon_{b+c+d} (1 + b + d - b\,c\,d)$,
which also differs from the
plateau of the inter-block gap, except when $c=1$.
This shows that the length of the inter-block gap in (\ref{Vfor3brRank2}) is exactly $1+a-b-c-d$ when $c \geq 2$. 
In other words
$G(\beta)$ completely describes the inter-block gap when $c \geq 2$.

By direct calculation, and under suitable restrictions, one finds that when 
$\beta = \sigma_1^{a} \sigma_2^{b} \sigma_1^{c} \sigma_2^{d} \sigma_1^{e} \sigma_2^{f}$, and $w^* = w(\beta)- a$,
we have a similar pattern as in Cor.~\ref{corRank2Vpos3br}:
$$V_{ \widehat{ \beta } }^* = B_{1}(\beta) + t^{w^*+1} G(\beta) + t^{a+2}B_{2}(\beta),$$
with $G(\beta) = \epsilon_{w^*} \{-( b  +  d  + f) + (bcf +bef + def +bcd ) -bcdef \} t^{1+w^*} A_{a+1-w^*}$.
This suggests that it would be helpful to introduce some terminology to describe the terms in $G(\beta)$
in a way that applies to the general case.

\begin{defn}
Given a three-braid word, $\beta = \prod_{i=1}^{r}\sigma_1^{e_{2i-1}} \sigma_2^{e_{2i}} $, a
$j$-fold product of twist exponents formed by starting with a twist exponent of $\sigma_2$
and alternately choosing a twist exponent of $\sigma_1$ and then $\sigma_2$, all with ascending subscripts 
is called a $j$-form of $\beta$, i.e. $\prod_{i=1}^{j} e_{k_i}$, with $k_i < k_{i+1}$ and $k_i \equiv i+1 \mbox{ mod } 2$
for all $i$. A $j$-form may also be viewed as an element of $\mathbb{Z}[e_1,\ldots,e_{2r}]$.
The sum of all $j$-forms of $\beta$ is denoted $m_j(\beta)$, viewed as a polynomial, 
and $m_j(\beta,0)$ when evaluated to an integer.

The alternating sum, $\sum_{k=1}^{r} \epsilon_{k} m_{2k-1}(\beta)$, is denoted $M(\beta)$, viewed as a polynomial, 
and $M(\beta,0)$ when evaluated to an integer.

\end{defn}

The number of $j$-forms of $\beta$ with $j$ odd is 
$\left( \begin{array} 
{cc} r+(j-1)/2 \\ j
\end{array} \right)$. This is trivially true for $j=1$. 
For $j \geq 3$, any $j$-form either ends in 
$e_{2k-1}e_{2r}$ with $k \leq r$, or ends in $e_{k}$, with $k \leq 2r-2$ and even.
By induction, in the first case for each $k\in{[(j+1)/2,r]}$, the cardinality is
$c_k=\left( \begin{array} 
{cc} k-1+(j-3)/2 \\ j-2
\end{array} \right)$,  while
$\kappa=\left( \begin{array} 
{cc} r-1+(j-1)/2 \\ j
\end{array} \right)$ gives the number of $j$-forms for the second case.
These cases are disjoint so the total is $\sum_{k=(j+1)/2}^{r} c_k + \kappa$, which has the value asserted.

The only non-trivial permutation of the twist exponents that preserves $M(\beta)$ is the one that ''reverses'' $\beta$ to
$\overset{\leftharpoonup}{\beta}$, where we may define a sort of reversed braid word, 
$\overset{\leftharpoonup}{\beta}= \sigma_1^{e_1} \sigma_2^{e_{2r}} \prod_{j=2}^{r} \sigma_1^{e_{2r-2j+3}} \sigma_2^{e_{2r-2j+2}} $
for $r \geq 2$. This is easily verified when $r=2$. 
Otherwise any such permutation, $\pi$, must preserve 1-forms, so viewing $\pi$ as a permutation of subscripts,  
$\pi$ permutes the even subscripts, permutes the odd subscripts, and fixes 1. 
We may calculate the number of 3-forms containing $e_{2k}$ %, say $c_{3,2k}$ 
as 
$%c_{3,2k} =
\left( \begin{array} 
{cc} k \\ 2
\end{array} \right)
+
\left( \begin{array} 
{cc} r-k+1 \\ 2
\end{array} \right)$.
This is maximal exactly when $k=1$ or $k=r$, hence $\pi$ must map $\{2,2r\}$ to itself.
We must have $\pi(2k-1)$ between $\pi(2j)$ and $\pi(2k+2\delta)$ for any $j<k$ and $\delta= 0, \ldots, r-k$.
When $\pi(2)=2$, we must have $\pi(2k-1) < \pi(2k+2\delta)$ for any $\delta= 0, \ldots, r-k$,
which easily leads to $\pi$ as the identity.
When $\pi(2)=2r$, we must have $\pi(2k-1) > \pi(2k+2\delta)$ for any $\delta= 0, \ldots, r-k$.
which easily leads to $\pi(\beta) = \overset{\leftharpoonup}{\beta}$.

%%%%%%%%%%%%%%%%%%%%%% General positive 3br CASE %%%%%%%%%%%%%%%%%%%%%%%%%%%%%%%%
%%%%%%%%%%%%%%%%%%%%%% General positive 3br CASE %%%%%%%%%%%%%%%%%%%%%%%%%%%%%%%%
%%%%%%%%%%%%%%%%%%%%%% General positive 3br CASE %%%%%%%%%%%%%%%%%%%%%%%%%%%%%%%%

The following proposition describes conditions under which the blocks in the coefficient vector
are separated, along with some of their properties.
\begin{prop}
Given a positive three-braid word, $\beta = \prod_{j=1}^{r}\sigma_1^{e_{2j-1}} \sigma_2^{e_{2j}} $,
set $w=w(\beta)$ and $w^* = w- e_1$. When $r \geq 2$ and $e_1 \geq w^*$, 
we have the following partitioning of coefficients:

\begin{eqnarray}
V_{ \widehat{ \beta } }^* &=& B_{1}(\beta) + t^{w^*+1} G(\beta) + t^{e_1+2}B_{2}(\beta), \mbox{ with }
B_{1}(\beta), B_{2}(\beta) \in \mathbb{Z}[t],
\label{Vforpos3brSpecial} \\
%%%
G(\beta) &=& \epsilon_{w^*} M(\beta,0) A_{e_1+1-w^*} \in \mathbb{Z}[t].
\label{Gforpos3brSpecial}
\end{eqnarray}
Here $\deg B_{1}(\beta) \leq w^*$, and $B_{1}(\beta)$ depends on the parity of $e_1$, but not on the actual value of $e_1$.
$B_{2}(\beta)$ is independent of $e_1$ with $\deg B_{2}(\beta) \leq w^*-3$.

The coefficients of $G(\beta)$ are independent of $e_1$, and the width of
the inter-block gap depends linearly on $e_1$.
\label{propVpos3brSpecial}
\end{prop}

Note that as in the discussion following Cor.~\ref{corRank2Vpos3br}, it is possible that some of the
high order coefficients of $B_{1}(\beta)$ and some of the low order coefficients of 
$B_{2}(\beta)$ could be incorporated into the inter-block gap for some special braid words $\beta$.

\begin{proof}
Cor.~\ref{corRank2Vpos3br} establishes the result for $r=2$ and we will use induction.
For $r \geq 3$,
apply (\ref{Jonesrecur}) to twist exponents
$e_{2r-3}$ and $e_{2r-2}$ of $\beta$.
Use the braid relations to obtain four terms each associated with a three-braid word
with $2r-2$ twist exponents. 
%With
%$\beta = \sigma_1^{e_{1}}  \ldots \sigma_2^{e_{2r-4}} \sigma_1^{e_{2r-3}} \sigma_2^{e_{2r-2}} \sigma_1^{e_{2r-1}} \sigma_2^{e_{2r}}$, 
The four braid words are 
$\beta_1 = \sigma_1^{e_{1}}  \ldots \sigma_2^{e_{2r-4}+e_{2r-1}} \sigma_1^{1} \sigma_2^{e_{2r}+1}$, and
$\beta_2 = \sigma_1^{e_{1}}  \ldots \sigma_2^{e_{2r-4}} \sigma_1^{e_{2r-1}+1} \sigma_2^{e_{2r}}$, and
$\beta_3 = \sigma_1^{e_{1}}  \ldots \sigma_2^{e_{2r-4}+1}  \sigma_1^{e_{2r-1}} \sigma_2^{e_{2r}}$, and
$\beta_4 = \sigma_1^{e_{1}}  \ldots \sigma_2^{e_{2r-4}} \sigma_1^{e_{2r-1}} \sigma_2^{e_{2r}}$.
We have
% see 2/27/2011 version
\begin{eqnarray*}
V_{ \widehat{ \beta } }^* 
&=&
A_{e_{2r-3}}\, A_{e_{2r-2}} V_{\widehat{ \beta_1  }}^*   
 + t A_{e_{2r-3}}\, A_{e_{2r-2}-1} V_{\widehat{ \beta_2 }}^* \\
&& + t A_{e_{2r-3}-1}\,A_{e_{2r-2}} V_{\widehat{ \beta_3 }}^*   
  + t^{2}\, A_{e_{2r-3}-1}\, A_{e_{2r-2}-1} V_{\widehat{ \beta_4 }}^*. 
\end{eqnarray*}

Now use the induction hypothesis, (\ref{Vforpos3brSpecial}), (\ref{Gforpos3brSpecial}).
The one challenge is to verify the claim for the inter-gap constant, i.e. to 
show that we have $M(\beta) = \omega$ where

\begin{eqnarray*}
\omega
&=& e_{2r-3}e_{2r-2} M( \beta_1 )   - e_{2r-3}(e_{2r-2}-1) M( \beta_2 ) \\
&& - (e_{2r-3}-1)e_{2r-2}  M( \beta_3 )  + (e_{2r-3}-1)(e_{2r-2}-1) M( \beta_4 ).
\end{eqnarray*}

The following identity may be used to complete this calculation. 

With $\eta_{k} = \prod_{j=1}^{k}\sigma_1^{f_{2j-1}} \sigma_2^{f_{2j}} $
and $\gamma = \eta_{r}$, we have
 
\begin{eqnarray*}
M( \gamma ) &=&
-f_{2r} -  \sum_{j=2}^{r}  M( \eta_{j-1} )f_{2j-1} f_{2r} + M( \eta_{r-1} ).
\end{eqnarray*}
\end{proof}

\begin{examp}
The $12n_{0242}$ knot, $\widehat{\sigma_1^{7} \sigma_2^{1} \sigma_1^{2} \sigma_2^{2} }$, has 
$V_{12n_{0242}}^* =   1+t^2   - t^6 A_3 $.

This exhibits a case where $\deg B_{1}(\beta) < w^*$ and $B_{2}(\beta)=0$.

The $12n_{0574}$ knot, $\widehat{\sigma_1^{6} \sigma_2^{2} \sigma_1^{3} \sigma_2^{1} }$, has
$V_{12n_{0574}}^* =   1+t^2-t^4+t^5-2t^6 + 3t^7 -2t^8+t^9 -t^{10}$.
Here $B_{1}(\beta) = 1+t^2-t^4+t^5-2t^6$ has $\deg B_{1}(\beta) =w^*$ 
and $B_{2}(\beta)=-2+t -t^{2}$ has $\deg B_{2}(\beta) <w^*-3$.
\end{examp}

%%%%%%%%%%%%%%%%%%%%%%%%%%%%%%%%%%%%%%%%%%%%%%%%%%%%%%%%%%%%%%%%%%%%%%%%%%%%%%%%%%%%%%%%%%%%
%%%%%%%%%%%%%%%%%%%%%%%%%%%%%%%%%%%%%%%%%%%%%%%%%%%%%%%%%%%%%%%%%%%%%%%%%%%%%%%%%%%%%%%%%%%%
%%%%%%%%%%%%%%%%%%%%%%%%%%%%%%%%%%%%%%%%%%%%%%%%%%%%%%%%%%%%%%%%%%%%%%%%%%%%%%%%%%%%%%%%%%%%

\subsection{Twisting three braids of a three-braid link}
\label{SSTwist3}
The general results obtained by A.~Champanerkar and I.~Kofman, \cite{22}, lead to an explicit 
description of the blocks in the coefficient vector when full twists are added to a
three-braid.

An immediate consequence of Prop.~\ref{PropJones3brfulltwists} is the following:
\begin{cor}
For any three-braid of the form $\beta = [1,3]^{3a} \sigma_{1}^{x} \sigma_{2}^{y}$, we have:

\begin{eqnarray}
V_{\widehat {\beta}}^*
&=&  \epsilon_{x+y}   ( 1+t^{2} )  + 
t^{3a}  ( t^{x+y+2} + \epsilon_{x} G t^{y+1} + \epsilon_{y} G t^{x+1} + \epsilon_{x+y}t^2 )/(1+t)^2    \,, \nonumber \\
\label{Jones3brfulltwists1x2y3a} \\
%&=&  \epsilon_{x+y}   ( 1+t^{2} )  + 
%t^{3a}\{ \epsilon_{x+y+1}  ( 1+t^{2}) + (\epsilon_{x}+ t^2 A_{x-1})(\epsilon_{y}+ t^2 A_{y-1}) \}  \,, \nonumber \\
%
&=& \epsilon_{x+y}   ( 1+t^{2} )  
+t^{3a+2} \{ \epsilon_{x+y+1}  + \epsilon_{y} A_{x-1}+ \epsilon_{x}  A_{y-1} + t^2 A_{x-1} A_{y-1} \}.\nonumber \\
\label{Jones3brfulltwists1x2y3aA}
\end{eqnarray}

When $\beta =[1,3]^{3a+1} \sigma_{1}^{x} \sigma_{2}^{y} \cong [1,3]^{3a} \sigma_{1}^{x+y+1} \sigma_{2}$, we have:

\begin{eqnarray}
V_{\widehat {\beta}}^*
&=&   \epsilon_{x+y}   ( 1+t^{2} ) -t^{3a+3}  ( t^{x+y-1} + \epsilon_{x+y})/(1+t)   \,, 
\label{Jones3brfulltwists1x2y3aPlus1} \\
&=& \epsilon_{x+y}   ( 1+t^{2} ) - t^{3a+3} A_{x+y-1}. 
\label{Jones3brfulltwists1x2y3aPlus1A} 
\end{eqnarray}

When $\beta =[1,3]^{3a+2} \sigma_{1}^{x} \sigma_{2}^{y} \cong [1,3]^{3a+3} \sigma_{1}^{x-1} \sigma_{2}^{y-1}$, we have:
\begin{eqnarray}
V_{\widehat {\beta}}^*
&=&  \epsilon_{x+y}   ( 1+t^{2} )
+ t^{3a+3} ( t^{x+y} + \epsilon_{x-1} G t^{y} + \epsilon_{y-1} G t^{x} + \epsilon_{x+y}t^2 )/(1+t)^2   \,. \nonumber \\
\label{Jones3brfulltwists1x2y3aPlus2}
\end{eqnarray}

\label{CorJones3brfulltwists}
\end{cor}

\subsubsection{The Jones polynomial for $\widehat{\sigma_{1}^{x} \sigma_{2}^{y} [1,3]^z}$ with $x,y,z \geq 0$}

In anticipation of Thm.~\ref{thmTlinkbr3} which shows that
the family of three-braid links,
$ \widehat{\sigma_{1}^{x} \sigma_{2}^{y} [1,3]^z} $ with 
$x,y \geq 0$, and $z \geq 3$ and $z \equiv 0 \mbox{ mod } 3$
coincides with the Lorenz links of braid index three,
we present the Jones polynomial for two-tier \mbox{T-links} and torus links as an
immediate consequence of Cor.~\ref{CorJones3brfulltwists}.

%%%%%%%%%%%%%%%%%%%%%% Jones polynomial for $L=\widehat{\sigma_{1}^{x} \sigma_{2}^{y} [1,3]^z}$ %%%%%%%%%%%%%%%%%%%%%%
%%%%%%%%%%%%%%%%%%%%%% Jones polynomial for $L=\widehat{\sigma_{1}^{x} \sigma_{2}^{y} [1,3]^z}$ %%%%%%%%%%%%%%%%%%%%%%
 
\begin{cor}
The Jones polynomial for the \mbox{T-link} $T((2,x),(3,3a))$, with $a, x>0$ is
\begin{eqnarray*}
V_{\widehat {[1,3]^{3a} \sigma_1^x} }&=& t^{(6a+x-2)/2} \{ \epsilon_{x} ( 1+ t^{2} ) + t^{3a+1}( \epsilon_{x} + t^{x} )  \} \,. 
\end{eqnarray*}
When $x=0$ we have $T(3,3a)$ and $V_{\widehat {[1,3]^{3a} } }= t^{3a-1} \{    1+ t^{2} + 2t^{3a+1}  \}$.
\bigskip

The Jones polynomial for the \mbox{T-link} $T((2,x),(3,3a+1))$, with $a, x>0$ is
\begin{eqnarray}
V_{\widehat {[1,3]^{3a+1} \sigma_{1}^{x} } }
&=&  t^{(6a+x)/2} \{ \epsilon_{x} ( 1+t^{2} ) -t^{3a+3} A_{x-1}    \} \,. \nonumber 
\end{eqnarray}
When $x=0$ we have $T(3,3a+1)$ and $V_{\widehat {[1,3]^{3a+1} } }= t^{3a} \{    1+ t^{2} -t^{3a+2}  \}$.
\bigskip

The Jones polynomial for the \mbox{T-link} $T((2,x),(3,3a+2))$, with $a \geq0, x>0$ is
\begin{eqnarray}
V_{\widehat {[1,3]^{3a+2} \sigma_{1}^{x} } }
&=& t^{(6a+x+2)/2} \{ \epsilon_{x} ( 1+t^{2} )  - t^{3a+3} A_{x+1}  \} \,. \nonumber 
\end{eqnarray}
When $x=0$ we have $T(3,3a+2)$ and $V_{\widehat {[1,3]^{3a+2} } }= t^{3a+1} \{  1+t^{2}   - t^{3a+3}  \}$.

\label{CorJones3brfulltwistsTlinks}
\end{cor}

A detailed description of the coefficient vectors for the first form of Cor.~\ref{CorJones3brfulltwists}, 
(\ref{Jones3brfulltwists1x2y3aA}), when $x,y,a \geq 0$ appears 
in Section~\ref{SSCoeffVectorszis0mod3}.

The following are a series of propositions that describe when two braid words of the form 
$\sigma_{1}^{x} \sigma_{2}^{y} [1,3]^z$, with $x,y,z \geq 0$, generate links that have the same Jones polynomial.
This leads to the interesting result, Theorem~\ref{ThmJonesCompleteInv},
that the Jones polynomial is a complete invariant for this class of links.
 
\begin{prop}
Suppose $L_i=\widehat{\sigma_{1}^{x_i} \sigma_{2}^{y_i} [1,3]^{z_i}}$ with $z_i \equiv 0 \mbox{ mod } 3$
and $x_i, y_i, z_i \geq 0$, for $i=1, 2$.

When $V_{L_1}=V_{L_2}$, we have $z_1=z_2$ and $\{x_1,y_1\} = \{x_2,y_2\}$.

When $z_1=z_2$ and $\{x_1,y_1\} = \{x_2,y_2\}$, we have $L_1=L_2$.
\label{propJonesUniqueTwistedTorusSum3a}
\end{prop}
\begin{proof}
The first assertion is clear. The second follows from the observation that $[1,3]^z = [3,1]^z$ is in the center of $B_3$
together with conjugacy by braid reflection.
%The second is clear for $z=0$ or when $0 \in \{x_1,y_1\}$. 
%Otherwise the second assertion follows from the observation that 
%$\sigma_{1}^{x} \sigma_{2}^{y} [1,3]^{z} \cong [1,3]\sigma_{1}^{x}[1,3]\sigma_{1}^{y} [1,3]^{z-2}$, which is 
%conjugate to $\sigma_{2}^{x}[1,3][1,3]\sigma_{1}^{y} [1,3]^{z-2}$ or 
%$\sigma_{2}^{x-1}[1,3]^3\sigma_{1}^{y-1} [1,3]^{z-2}$. The latter form is conjugate to 
%\sigma_{2}^{x}[1,3]^3\sigma_{1}^{y} [1,3]^{z-3} = \sigma_{2}^{x}\sigma_{1}^{y} [1,3]^{z} \cong \sigma_{1}^{y}\sigma_{2}^{x} [1,3]^{z}$.
%Thus $\sigma_{1}^{x} \sigma_{2}^{y} [1,3]^{z} \cong \sigma_{1}^{y}\sigma_{2}^{x} [1,3]^{z}$.
\end{proof}

Eq.~\ref{Jones3brfulltwists1x2y3aPlus1}, the prior comments, and Prop~\ref{propJonesUniqueTwistedTorusSum3a} give us

\begin{prop}
Suppose $L_i=\widehat{\sigma_{1}^{x_i} \sigma_{2}^{y_i} [1,3]^{z_i}}$ with $z_i \equiv 1 \mbox{ mod } 3$
and $x_i, y_i, z_i \geq 0$, for $i=1, 2$. 
Also assume $L=\widehat{\sigma_{1}^{x} \sigma_{2}^{y} [1,3]^{z}}$ with $z \equiv 0 \mbox{ mod } 3$
and $x, y, z \geq 0$.

When $V_{L_1}=V_{L_2}$, we have $z_1=z_2$ and $x_1+y_1 =x_2+y_2$.

When $z_1=z_2$ and $x_1+y_1 =x_2+y_2$, we have $L_1=L_2$.
\bigskip

When $V_{L_1} = V_{L}$, we have $\{x,y\}=\{1,1+x_1+y_1\}$ and $z+1=z_1$. 

When $\{x,y\}=\{1,1+x_1+y_1\}$ and $z+1=z_1$, we have $L_1=L$.

Hence, given $L_1$, it has a representation as $\sigma_{1}^{1+x_1+y_1} \sigma_{2}  [1,3]^{z_1-1}$.

\label{propJonesUniqueTwistedTorusSum3aplus1}
\end{prop}

Eq.~\ref{Jones3brfulltwists1x2y3aPlus2} and Props.~\ref{propJonesUniqueTwistedTorusSum3a}, \ref{propJonesUniqueTwistedTorusSum3aplus1}
give us

\begin{prop}
Suppose $L_i=\widehat{\sigma_{1}^{x_i} \sigma_{2}^{y_i} [1,3]^{z_i}}$ with $z_i \equiv 2 \mbox{ mod } 3$
and $x_i, y_i, z_i \geq 0$, for $i=1, 2$. 
Also assume $L=\widehat{\sigma_{1}^{x} \sigma_{2}^{y} [1,3]^{z}}$ with $z \equiv 0 \mbox{ mod } 3$
and $x, y, z \geq 0$.

When $V_{L_1}=V_{L_2}$ and $z_1 \geq z_2$, we either have 
\begin{enumerate}
\item $z_1=z_2$ and $\{x_1,y_1\} =\{x_2,y_2\}$, or
\label{Jonesequal1}
\item $z_1=z_2+3$, $\min(x_1,y_1)=0$, $\min(x_2,y_2)=2$, and \newline
 $\max(x_2,y_2)=4+\max(x_1,y_1)$.
\label{Jonesequal2}
\end{enumerate}
When either of these parameter conditions (\ref{Jonesequal1} or \ref{Jonesequal2}) are true, we have $L_1=L_2$.

When $V_{L_1} = V_{L}$, we either have 
\begin{enumerate}
\item $z=z_1+1$, $\{x+1,y+1\} = \{x_1,y_1\}$, or
\label{Jonesequal3}
\item $z=z_1-2$, $\max(x,y) = 3+ \max(x_1,y_1)$, $\min(x,y)=1$, $\min(x_1,y_1)=0$.
\label{Jonesequal4}
\end{enumerate}
When either of these parameter conditions (\ref{Jonesequal3} or \ref{Jonesequal4}) are true, we have $L_1=L$.

Hence, given $L_1$, it has a representation as either
\begin{enumerate}
\item $L = \widehat{ \sigma_{1}^{3+\max(x_1,y_1)} \sigma_{2}  [1,3]^{z_1-2} }$ when $\min(x_1,y_1)=0$, or
\item $L = \widehat{ \sigma_{1}^{x_1-1} \sigma_{2}^{y_1-1} [1,3]^{z_1+1} }$, when $x_1, y_1 \geq 1$.
\end{enumerate}

\label{propJonesUniqueTwistedTorusSum3aplus2}
\end{prop}

\begin{proof}
It is straightforward to verify the conditions implied when $V_{L_1}=V_{L_2}$. 
To prove $L_1=L_2$ when $z_1=z_2$ and $\{x_1,y_1\} =\{x_2,y_2\}$ is also straightforward.

To see that $L_1=L_2$ in case~\ref{Jonesequal2},
take the case that $L_1=\widehat{\sigma_2^x [1,3]^{z+3}}$, and $L_2=\widehat{\sigma_1^2 \sigma_2^{x+4} [1,3]^{z}}$
with $z \equiv 2 \mbox{ mod } 3$.
First, $\sigma_2^x [1,3]^{z+3} = \sigma_2^{x+1} \sigma_1 \sigma_2^2 [1,3]^{z+1} \cong \sigma_1 \sigma_2^{x+3} [1,3]^{z+1}$.
This is just $\sigma_1 \sigma_2^{x+4}\sigma_1 \sigma_2^2 \sigma_1 \sigma_2 [1,3]^{z-2}
\cong \sigma_1^2 \sigma_2^{x+4}\sigma_1 \sigma_2 \sigma_1 \sigma_2  [1,3]^{z-2}$.
It is easy to see that $\sigma_1^x [1,3]^{z+3} \cong \sigma_2^x [1,3]^{z+3}$,
and $\sigma_1^2 \sigma_2^{x+4} [1,3]^{z} \cong \sigma_1^{x+4} \sigma_2^2 [1,3]^{z}$,
hence all four combinations in case~\ref{Jonesequal2} represent the same link.

When $V_{L_1} = V_{L}$, and either $z=z_1+1$ or $z=z_1-2$, the remaining implications are trivial.
To show that these are the only two valid values for $z$ is straightforward.

To see that $L_1=L$ in case~\ref{Jonesequal3},
take the case that $L=\widehat{\sigma_1^x \sigma_2^y [1,3]^{z_1+1}}$, with $z_1 \equiv 2 \mbox{ mod } 3$.
First, $\sigma_1^x \sigma_2^y [1,3]^{z_1+1} = \sigma_1^{x} \sigma_1 \sigma_2 \sigma_1^{y} [1,3]^{z_1}$, and this is 
conjugate to $\sigma_1  \sigma_2\sigma_1^{x} \sigma_1 \sigma_2 \sigma_1^{y} \sigma_1  \sigma_2[1,3]^{z_1-2}
=\sigma_2^{x}  \sigma_1\sigma_2 \sigma_1 \sigma_2 \sigma_1^{y+1}  \sigma_2[1,3]^{z_1-2}
\cong \sigma_1^{y+1}  \sigma_2^{x+1}[1,3]^{z_1}$. It is easy to see that 
$\sigma_1^x \sigma_2^y [1,3]^{z_1+1} \cong \sigma_1^y \sigma_2^x [1,3]^{z_1+1}$, and
that $\sigma_1^{y+1}  \sigma_2^{x+1}[1,3]^{z_1} \cong \sigma_1^{x+1}  \sigma_2^{y+1}[1,3]^{z_1}$,
so all four combinations represent the same link.

To see that $L_1=L$ in case~\ref{Jonesequal4}, it is trivial to verify 
$\sigma_1^{3+a}\sigma_2^1 [1,3]^{z_1-2} \cong \sigma_1^{a} [1,3]^{z_1}$, 
$\sigma_1^1 \sigma_2^{3+a} [1,3]^{z_1-2} \cong \sigma_1^{a} [1,3]^{z_1}$, and
$\sigma_1^{a} [1,3]^{z_1}\cong \sigma_2^{a} [1,3]^{z_1}$.
\end{proof}

Combining Props.~\ref{propJonesUniqueTwistedTorusSum3aplus1} and \ref{propJonesUniqueTwistedTorusSum3aplus2}
and using similar techniques gives us

\begin{cor}
Suppose $L_i=\widehat{\sigma_{1}^{x_i} \sigma_{2}^{y_i} [1,3]^{z_i}}$ with $z_i \equiv i \mbox{ mod } 3$
and $x_i, y_i, z_i \geq 0$, for $i=1, 2$. 

When $V_{L_1} = V_{L_2}$, we either have 
\begin{enumerate}
\item $z_1=z_2+2$, $\{2,2+x_1+y_1\} = \{x_2,y_2\}$, or
\label{Jonesequal5}
\item $z_1=z_2-1$, $x_1+y_1 = 2+ \max(x_2,y_2)$, $\min(x_2,y_2)=0$.
\label{Jonesequal6}
\end{enumerate}

When either of these parameter conditions (\ref{Jonesequal5} or \ref{Jonesequal6}) are true, we have $L_1=L_2$.
\label{corJonesUniqueTwistedTorusSum3aplus2}
\end{cor}

%%%%%%%%%%%%%%%%%%%%%%%%%%%%% Theorem %%%%%%%%%%%%%%%%%%%%%%%%%%%%%%%%%%%%%%%%%%%%%%%%%%%%%
%%%%%%%%%%%%%%%%%%%%%%%%%%%%% Theorem %%%%%%%%%%%%%%%%%%%%%%%%%%%%%%%%%%%%%%%%%%%%%%%%%%%%%

The prior results imply that the Jones polynomial is a complete invariant and
identifies which are the torus links.
\begin{thm}
Suppose $L_i=\widehat{\sigma_{1}^{x_i} \sigma_{2}^{y_i} [1,3]^{z_i}}$ 
and $x_i, y_i, z_i \geq 0$, for $i=1, 2$. \newline
If $V_{L_1} = V_{L_2}$, we have $L_1 = L_2$.

$L_1$ has a representation as $\widehat{\sigma_{1}^{x} \sigma_{2}^{y} [1,3]^{z}}$ with $z \equiv 0 \mbox{ mod } 3$.

When $L_1$ is a torus link of braid index three, it has one of the following forms:
\begin{enumerate}
\item $x_1=0=y_1$ and $z_1 \geq 3$, so $L_1=T(3,z_1)$,
\item $x_1=1=y_1$ and $z_1 \geq 2$, so $L_1=T(3,z_1+1)$,
\item $\min(x_1,y_1)=1$, $\max(x_1,y_1)=3$, $z_1 \equiv 0 \mbox{ mod } 3$, $z_1 \geq 3$, so $L_1= T(3,z_1+2)$,
\item $\min(x_1,y_1)=2$, $\max(x_1,y_1)=4$, $z_1 \equiv 2 \mbox{ mod } 3$, so $L_1= T(3,z_1+3)$,
\item $x_1=2=y_1$, $z_1 \equiv 2 \mbox{ mod } 3$, so $L_1=T(3,z_1+2)$, 
\item $\min(x_1,y_1)=0$, $\max(x_1,y_1)=2$, $z_1 \equiv 1 \mbox{ mod } 3$, $z_1 \geq 4$, so $L_1= T(3,z_1+1)$.
\end{enumerate}
\label{ThmJonesCompleteInv}
\end{thm}

%%%%%%%%%%%%%%%%%%%%%%%%%%%%%%

\subsubsection{The coefficient vector for $L=\widehat{\sigma_{1}^{x} \sigma_{2}^{y} [1,3]^z}$ with $z \equiv 0 \mbox{ mod } 3$}
\label{SSCoeffVectorszis0mod3}
As in the prior section, we assume $x, y, z$ are non-negative for this section.

As Eq.~\ref{Jones3brfulltwists1x2y3a} is symmetric in $x,y$, we may assume $x \geq y$ with no loss in generality.
As Cor.~\ref{CorJones3brfulltwistsTlinks} has already described the case $y=0$, we may assume $x,y \geq 1$.

When $y=1$, the link is $L = T((2,x-1),(3,z+1))$ (see Section~\ref{TwoTierMis2Nis3}) and
$$V_L^* = \epsilon_{x+1}   ( 1+t^{2} )  -t^{z+3} A_{x-2}.$$ 
%$V^* = \epsilon_{x+1}   ( 1+t^{2} )  +t^{z+2} ( \epsilon_{x}  - A_{x-1} )$, which is just

When $y=2$, we have 
$$V_L^* =\epsilon_{x}   ( 1+t^{2} )  +t^{z+2} ( A_{x-1} + t^2 A_{x-1}).$$
For $x=2$ and $y=2$, $V_L^*$ is just $ 1+t^{2}  +t^{z+2} ( 1 + t^2 )$. \newline
For $x=3$ and $y=2$, we have $V_L^* = -( 1 + t^{2}) + t^{z+2}( -1+t-t^2+t^{3} )  $.

For $x \geq 4$ and $y=2$, we have
$$V_L^* = \epsilon_{x} ( 1+t^{2} )  +t^{z+2} \{ \epsilon_{x}(1-t) + 2t^2 A_{x-3} - t^{x-1}+ t^x \}.$$
This expression also applies when $x=2$ or $x=3$, but is less intuitive.

For $x,y \geq 3$, Cor.~\ref{CorJones3brfulltwists}, Eqs.~\ref{AProductUsingL} and \ref{ARecurwz} give us
\begin{eqnarray}
V_L^* &=& \epsilon_{x+y}   ( 1+t^{2} )  + t^{z+2} B_2, \mbox{ with } \label{VstarforRank1wfulltwist} \\
B_2 &=& \epsilon_{x} L(y-2) + \epsilon_{x+1} (y-1) t^{y-1} + \epsilon_{y} y t^y A_{x-y-1} \nonumber \\
&&+ \epsilon_{y+1} (y-1) t^{x-1} +  t^x R(y-2). \label{B2forRank1wfulltwist}
\end{eqnarray}

%A_x A_y &=& \epsilon_x  L( y-2 )  - (y) \epsilon_y t^{y-1} A_{x-y+1} +  t^x R(y-2). \label{AProductUsingL}
%A_{w} &=& t^{z} A_{w-z} + \epsilon_{w-z} A_{z}, \label{ARecurwz} 

% $A_{x-1} = t^{y} A_{x-1-y} + \epsilon_{x-1-y} A_{y}$
Note that when $x=y$, the expression for $B_2$ becomes
$$ B_2 = \epsilon_{x} L(x-2) + \epsilon_{x+1} (x-2) t^{x-1} +  t^x R(x-2).$$
Hence a graph of the absolute values of the coefficients in the second block
for $x=y$ looks like an inverted W.

%%%%%%%%%%%%%%%%%%%%%%%%% Braid Attributes %%%%%%%%%%%%%%%%%%%%%%%%%%%%%%%%%%%%%%%%%
%%%%%%%%%%%%%%%%%%%%%%%%% Braid Attributes %%%%%%%%%%%%%%%%%%%%%%%%%%%%%%%%%%%%%%%%%
%%%%%%%%%%%%%%%%%%%%%%%%% Braid Attributes %%%%%%%%%%%%%%%%%%%%%%%%%%%%%%%%%%%%%%%%%

\section{Braid properties related to twisted n-braid links, \mbox{T-links} and their symmetries}
\label{sectionBraidPropsTlinks}
In this section,
Prop.~\ref{PropBraidRel} describes some relations that are useful in working with twisted n-braid links and \mbox{T-links}.
In particular, Prop.~\ref{PropBraidRel} leads to some symmetry results, 
Cors.~\ref{corktwistednbrFinalParamsSymmetry} and \ref{corktwistednbr2FinalParamsSymmetry}, 
that describe some cases when two such n-braids close to form the same link.

%The observation that $\sigma_a^r {[1,n]} = {[1,n]} \sigma_{a-1}^r$ for $2 \leq a < n$ and any integer $r$
%readily leads to the fact that $\sigma_a^r {[1,n]}^n = {[1,n]}^n \sigma_a^r $.
%Indeed, for $1 \leq a \leq n-1$,
%the left represents $ {[1,n]}^{a-1} \sigma_1^r {[1,n]}^{n-a+1}$, while the right represents 
%$ {[1,n]}^{a+1} \sigma_{n-1}^r {[1,n]}^{n-a-1}$.
%Hence it suffices to show that $  \sigma_1^{\pm 1} {[1,n]}^{2} = {[1,n]}^{2} \sigma_{n-1}^{\pm 1}$.
%This is a straightforward induction proof.
%The same technique shows that $\sigma_a {[n,1]}^n = {[n,1]}^n \sigma_a $.
%It is also known that ${[n,1]}^n = {[1,n]}^n$, \cite{12}, which generates the center of the braid group, \cite{10}.

Note that Eq.~\ref{Tlinkrel1} below also follows from Ex.~2 \cite{12} and duality.

\begin{prop}
Assume $n \geq 2$,  and $1 \leq d \leq n$. We have the following:
\begin{eqnarray}
%[1,n]^2 = \sigma_1 [1,n] [1,n-1]  & and & [n,1]^2 = [n-1,1]  [n,1] \sigma_{1},  \\
%
{[1,n]}^d = [d,1] [1,n] [1,n-1]^{d-1}  & and & [n,1]^d = [n-1,1]^{d-1} [n,1] [1,d]. 
\label{BraidRelReduction}    
%
%{[1,m+1]}^d = [d,1] [1,m+1] [1,m]^{d-1}  & and & [m+1,1]^d = [m,1]^{d-1} [m+1,1] [1,d].
\end{eqnarray}

Assume $1 \leq \rho \leq m$ and $2 \leq m+x$, with $x \geq 0$. 
With $[1,m+x] \in B_{m+x}$ and $\gamma,\, [\rho,1]^x  [1,m]^{\rho} \in B_m$, we have
\begin{eqnarray}
%\widehat{ {[1,m+x]}^{\rho} } = \widehat{  {[\rho,1]}^x {[1,m]}^{\rho} } 
%& and & \widehat{ {[m+x,1]}^{\rho} } = \widehat{   {[m,1]}^{\rho} {[1,\rho]}^x }, 
%\nonumber \\
%
\widehat{ \gamma {[1,m+x]}^{\rho} } = \widehat{  \gamma {[\rho,1]}^x {[1,m]}^{\rho} }
& and & \widehat{ \gamma {[m+x,1]}^{\rho} } = \widehat{  \gamma  {[m,1]}^{\rho} {[1,\rho]}^x}.
\label{brred2}
\end{eqnarray}

Equation~\ref{brred2} may be applied to \mbox{T-links} as:
\begin{eqnarray}
T((m,y),(n,\rho)) &=& \widehat{{[\rho,1]}^{n-m} {[1,m]}^{\rho+y}}, 
\mbox{ for } 2 \leq \rho \leq m \leq n \mbox{ and } 1 \leq y, \label{Tlinkbrred1} \\
T((2,y),(n,\rho)) &=& T((2,y),(\rho, n))
\mbox{ for } 2 \leq \rho, n \mbox{ and } 1 \leq y, \label{Tlinkbrred1mis2} \\
T((m,y)(n,m)) &=& T(m,n+y),  \mbox{ for } 2 \leq m < n \equiv 0 \mbox{ mod } m. \label{Tlinkrel1}
\end{eqnarray}

\label{PropBraidRel}
\end{prop}
 
\begin{proof}
Eq.~\ref{BraidRelReduction} is readily proven by induction on $d$.
Eq.~\ref{brred2} follows from (\ref{BraidRelReduction}) and the 
use of Markov destabilization to reduce the number of strands.

To establish (\ref{Tlinkbrred1mis2}), note that by duality we have 
$T( (2,y ), (n,\rho) ) = T( (\rho,n-2 ), (y+\rho,2) )$. When $\rho, n >2$,
Eq.~\ref{Tlinkbrred1} applied to the latter form tells us these are just 
$\widehat{{[2,1]}^{y} {[1,\rho]}^{n}}= T((2,y),(\rho,n))$.
When $\rho=2 < n$, we have $T(2,n+y)$ and the result holds.
As the assertion is symmetric in $\rho, n$, the result holds.
\end{proof}

Eq.~\ref{brred2} leads to a symmetry result that applies to twisted $n$-braids. 
For \mbox{T-links} the result is somewhat analogous to Cor.~3, \cite{12}.
The cited corollary states that 
$T((r_{1},s_{1}),\ldots,(r_{k-1},s_{k-1}),\mathbf{(r_{k},s_{k})}) = T((r_{1},s_{1}),\ldots,(r_{k-1},s_{k-1}),\mathbf{(s_{k},r_{k})})$
when $s_i | r_{i}$ for $i=1,\ldots,k-1$ and $r_{k-1} \leq s_{k}$.
The following corollary gives the same symmetry when the restrictions are removed from the low index parameters
and are transferred to the final pair of parameters.

\begin{cor}
When $\gamma \in B_a$, with $a \leq s \leq r$ and $s | r $, we have
$\widehat{\gamma [1,r]^{s}} = \widehat{\gamma [1,s]^{r}}$.

It follows that when $r_{k-1} \leq s_k \leq r_k$ and $s_{k} | r_{k}$, we have
\begin{eqnarray*}
T((r_{1},s_{1}),\ldots,(r_{k-1},s_{k-1}),(r_{k},s_{k})) = T((r_{1},s_{1}),\ldots,(r_{k-1},s_{k-1}),(s_{k},r_{k})).
\end{eqnarray*}
\label{corktwistednbrFinalParamsSymmetry}
\end{cor}

A further set of symmetry results implied by Prop.~\ref{PropBraidRel} is the following
\begin{cor}
When $\gamma \in B_a$, with $a \leq m \leq n_1,n_2$ and $n_1+y_1 = n_2+y_2$ and $[n_1]_m = [n_2]_m$, we have
$\widehat{\gamma [1,m]^{y_1} [1,n_1]^{m}} = \widehat{\gamma [1,m]^{y_2} [1,n_2]^{m}}$.
Applied to \mbox{T-links}, $L_i=T((r_{1},s_{1}),\ldots,(r_{k-2},s_{k-2}),(m,y_i),(n_i,m))$ for $i=1,2$ 
and $r_{k-2} \leq m$, we have $L_1 = L_2$.
In particular, for two-tier \mbox{T-links}, we have
\begin{eqnarray}
T((m,y_1),(n_1,m)) &=& T((m,y_2),(n_2,m)).
\label{Eq2tierTlinksMisRhoSymmetry1}
\end{eqnarray}

When $n_1+y_1 = n_2+y_2$ and instead $[n_1]_m = [y_2]_m$, we also have
\begin{eqnarray}
T((m,y_1)(n_1,m)) = T((m,y_2)(n_2,m)).
\label{Eq2tierTlinksMisRhoSymmetry2}
\end{eqnarray}

In particular, when $m <y$ we have $T((m,y)(n,m)) = T((m,n)(y,m))$.
By Cor.~3, \cite{12}, this is also true when $m=y$.
\label{corktwistednbr2FinalParamsSymmetry}
\end{cor}

\begin{proof}
%\widehat{ \gamma {[1,m+x]}^{\rho} } = \widehat{  \gamma {[\rho,1]}^x {[1,m]}^{\rho} }
By (\ref{brred2}), we have $\widehat{\gamma [1,m]^{y_i} [1,n_i]^{m}} = \widehat{\gamma [1,m]^{y_i}  {[m,1]}^{n_i-m} {[1,m]}^{m} }$,
and this is $\widehat{\gamma [1,m]^{[y_i]_m}  {[m,1]}^{[n_i]_m} {[1,m]}^{y_i-[y_i]_m+n_i-[n_i]_m} }$,
which is thus the same link for $i=1,2$.

Eq.~\ref{Eq2tierTlinksMisRhoSymmetry2}: First apply duality to $T((m,y_2)(n_2,m))$ and then use 
Eq.~\ref{Eq2tierTlinksMisRhoSymmetry1}.
\end{proof}

Braid reflection implies
${[m,1]}^{{[n]}_m} {[1,m]}^{{[y]}_m}$ is conjugate to ${[m,1]}^{{[y]}_m} {[1,m]}^{{[n]}_m}$.

Eq.~\ref{Eq2tierTlinksMisRhoSymmetry1} tells us that for a fixed choice of $S$ and $\delta \in {[0,m-1]}$
the family of \mbox{T-links} given by 
$\{T((m,S-\lambda m - \delta),(\lambda m + \delta,m)): \lambda \in {[1, \lfloor (S- \delta-1)/m \rfloor]} \}$
represent a single link.

Furthermore $T((m,S-\lambda m - \delta),(\lambda m + \delta,m)) = T((m,\lambda m + \delta ),(S-\lambda m - \delta,m))$
for $\lambda \leq \lfloor (S- \delta-m)/m \rfloor$ by Eq.~\ref{Eq2tierTlinksMisRhoSymmetry2},
or by the prior observation plus duality.
These two families have disjoint parameter values except when $S-\lambda m - \delta = \lambda' m + \delta$,
i.e. $S - 2\delta=m(\lambda + \lambda')$ for suitable choices of 
$\lambda, \lambda' \in {[1,\,\lfloor (S- \delta-m)/m \rfloor]}$.

\begin{examp}
To illustrate the first observation, suppose $S=37$, $m=7$, $\delta=3$.  The set 
$\{ T((7,27),(10,7)); T((7,20),(17,7)); T((7,13),(24,7)); T((7,6),(31,7)) \}$ represents
a single \mbox{T-link}.

The second observation shows that
$T((7,27),(10,7)) = T((7,10),(27,7))$ and 
$T((7,20),(17,7)) = T((7,17),(20,7))$ and
$T((7,13),(24,7)) = T((7,24),(13,7))$.

These two families are disjoint, since $31 = 7 (\lambda + \lambda')$ has no integral solution.
\end{examp}

Cor.~8, \cite{12}, gives a formula for the braid index of a \mbox{T-link} in terms of 
the form, $L=T((r_1, s_1), \ldots, (r_k, s_k))$ and its dual form 
$T((\overline{r_1}, \overline{s_1}), \ldots, (\overline{r_k}, \overline{s_k}))$. With
$i_0=\min\{i: r_i \geq \overline{r_{k-i}} \} $, and 
$j_0=\min\{j: \overline{r_j} \geq r_{k-j} \}$, Cor.~8, \cite{12} gives 
$b(L)=\min(r_{i_0}, \overline{r_{j_0}})$.
The following is a simple but useful observation:

\begin{prop}
Suppose $L=T((r_1, s_1), \ldots, (r_k, s_k))$.

When $r_{i_0} = \overline{r_{k-i_0}}$, we have $j_0 = k-i_0$ and $b(L)= r_{i_0}$.

When $r_{i_0} > \overline{r_{k-i_0}}$, we have $j_0 = 1+k-i_0$ and $b(L)=\min(r_{i_0}, \overline{r_{1+k-i_0}})$.

Furthermore, $r_i \geq i+1$ and $\overline{r_i} \geq i+1$.
\label{PropBraidIndexTlink}
\end{prop}

\begin{cor}
Suppose $L=T((r_1, s_1), \ldots, (r_k, s_k))$. We have $k \leq 2 b(L)-2$.

When $k=2 b(L)-2$ we have $i_0=b(L)-1=j_0$ and $r_{b(L)-1}=b(L)=\overline{r_{b(L)-1}}$;
indeed $r_i = i+1 = \overline{r_{i}}$ for $i \leq b(L)-1$.
\label{cormaxtiers}
\end{cor}
\begin{proof}
We must have $i_0 \leq b(L)-1$ or $j_0 \leq b(L)-1$ by Cor.~8, \cite{12} and Prop.~\ref{PropBraidIndexTlink}.

First consider the case $j_0 \geq b(L)$, so that $i_0 \leq b(L)-1$.
As $b(L)=r_{i_0} \geq \overline{r_{k-i_0}} \geq k-i_0+1$, we have $k \leq 2b(L)-2$. 
Now $k=2b(L)-2$ implies $i_0 = b(L)-1$ and $b(L)= \overline{r_{k-i_0}}$, so
$j_0=k-i_0$ by Prop.~\ref{PropBraidIndexTlink}. Thus $j_0= b(L)-1$ contrary to assumption.
A similar argument applies when $i_0 \geq b(L)$, so we are left with the cases that $i_0,j_0 \leq b(L)-1$.

When $r_{i_0} > \overline{r_{k-i_0}}$, we have $j_0 = 1+k-i_0$, i.e. $k=i_0+j_0-1 \leq 2 b(L)-3$.
When $r_{i_0} = \overline{r_{k-i_0}}$, we have $j_0 = k-i_0$, so that $k \leq 2 b(L)-2$, with equality exactly when $i_0=j_0=b(L)-1$.
\end{proof}

Cor.~\ref{cormaxtiers} shows that \mbox{T-links} with braid index two must have two or fewer tiers;
those with braid index three must have four or fewer tiers.
For a fixed choice of braid index, $b$, the \mbox{T-link} with the maximum number of tiers, $2 b-2$, has the 
form $T((2,s_1), (3,s_2),\ldots, (b,s_{b-1}), (r_{b},1),\ldots, (r_{2 b-3},1),(r_{2 b-2},2))$.

%%%%%%%%%%%%%%%%%%%%%%%%%%%%%%%%%%%%%%%%%%%
\section{The Jones polynomial and other properties of T-links}
\label{SectionTlinks}
\subsection{The Jones polynomial for two tier \mbox{T-links} with $b(L) \leq 3$}
By Cor.~8 \cite{12}, in order for the braid index to be two, a two tier \mbox{T-link} must have the form
$T( (2,y), (n,2))$.
Eq.~\ref{Tlinkbrred1mis2} tells us that these are just $T(2,n+y)$. 

In order for the braid index to be three, a two tier \mbox{T-link} must have one of the
following forms, $T( (m,y ), (n,\rho) )$:
\begin{enumerate}
\item $T( (2,y ), (3,\rho) )$, with $3 \leq \rho$, (see Section~\ref{TwoTierMis2Nis3}),
\item $T( (2,y ), (n,3) )$, (see Section~\ref{TwoTierMis2Rhois3}),
\item $T( (3,y), (n,3))$, (see Section~\ref{TwoTierMis3Rhois3}),
\item $T( (3,y), (n,2))$, (see Section~\ref{TwoTierMis3Rhois2}),
\item $T( (m,1), (n,2))$, with $3 < m$, (see Section~\ref{TwoTierYis1Rhois2}),
\end{enumerate}

\subsubsection{The Jones polynomial for two tier \mbox{T-links}  with $m=2$, $\rho=3$}
\label{TwoTierMis2Rhois3}
Eq.~\ref{Tlinkbrred1mis2} applied to $T( (2,y ), (n,3) )$ tells us these are just $T((2,y),(3,n))$.
This is analyzed in Section~\ref{TwoTierMis2Nis3}.

\subsubsection{The Jones polynomial for two tier \mbox{T-links}  with $m=3$, $\rho=2$}
\label{TwoTierMis3Rhois2}
By duality we have $T( (3,y ), (n,2) ) = T( (2,n-3 ), (y+2,3) )$.
Eq.~\ref{Tlinkbrred1mis2} applied to the second form tells us these are just $T( (2,n-3 ), (3,y+2) )$.
This is analyzed in Section~\ref{TwoTierMis2Nis3}.

\begin{examp}
\label{ExhyperbolicTlinkmis3rhois2}
$\mathbf{k}7_{39} = T((3,3),(11,2))$ is a hyperbolic Lorenz knot, \cite{12}, \cite{23}.
% verified Jones polynomial against \cite{23}; 
%arXiv version has wrong representation - gives $T(11,2,3,1)$  which corresponds to $T((3,1),(11,2))$ in this notation
\end{examp}

\subsubsection{The Jones polynomial for two tier \mbox{T-links}  with $m=3$, $\rho=3$}
\label{TwoTierMis3Rhois3}
By duality we have $T( (3,y ), (n,3) ) = T( (3,n-3 ), (y+3,3) )$.
Eq.~\ref{Tlinkrel1} tells us these are just $T(3,n+y)$ when $3|n$ or $3|y$.

Restating (\ref{Vfor3brTorusLink3comp}) and (\ref{Vfor3brTorusKnots}) 
for the case that $3|n$ or $3|y$, we have $T( (3,y ), (n,3) ) = T(3,n+y)$ and 
\begin{eqnarray*}
V_{T( (3,y ), (n,3) )}(t) &=& t^{n+y-1} \{1+ t^2+ 2t^{n+y+1} \} , \mbox{ when } 3|n \mbox{ and } 3|y, \\
&=& t^{n+y-1} \{1+ t^2 - t^{n+y+1}   \}   , \mbox{ when } 3 \not | (n+y).
\end{eqnarray*}

%%%%%%%%%%%%%%%%%%  General Case %%%%%%%%%%%%%%%%%%%%%%%%%%%%%%%%%%%%%%%%%%%%%%%%%%%

To calculate the Jones polynomial when $n,y \not \equiv 0 \mbox{ mod }3$, note that (\ref{Tlinkbrred1}) tells us that 
$T((3,y),(n,3)) = \widehat{{[3,1]}^{n-3} {[1,3]}^{3+y}}$.
The braid properties discussed in Section~\ref{sectionBraidPropsTlinks} allow us to write
$T((3,y),(n,3)) = 
%\widehat{{[3,1]}^{{[n]}_3} {[1,3]}^{y+3a}}= 
\widehat{   {[3,1]}^{{[n]}_3} {[1,3]}^{y+n-{[n]}_3}}$.
%where $a= \lfloor n/3 \rfloor$ .

When ${[n]}_3 = 1={[y]}_3$, we have $T((3,y),(n,3)) = \widehat{ [3,1][1,3] [1,3]^{y-[y]_3+n-[n]_3}  } $, and this is 
$\widehat{ \sigma_1^2 \sigma_2^2 [1,3]^{y-[y]_3+n-[n]_3}  } $, so we may apply (\ref{Jones3brfulltwists1x2y3a}) with $w=2y+2n$ to get
\begin{eqnarray*}
V_{T((3,y),(n,3))} &=& t^{y+n-1} \{ ( 1 + t^{2})  + t^{y+n} ( 1+t^{2}) \}.
\end{eqnarray*}
This class contains no torus links.

%%%%%%%%%%%%%%%%%%%%%%%%%%%%%%%%%%%%%%%%%%

When ${[n]}_3 = 1$ and ${[y]}_3=2$, 
we have $T((3,y),(n,3)) = \widehat{ \sigma_1^5 \sigma_2^1  [1,3]^{y-[y]_3+n-[n]_3}  } $, 
so we may apply (\ref{Jones3brfulltwists1x2y3a}) with $w=2y+2n$ to get
\begin{eqnarray*}
V_{T((3,y),(n,3))} &=& t^{y+n-1} \{  ( 1 + t^{2}) - t^{y+n} ( 1 -t+ t^2)  \}.
\end{eqnarray*}
The same result applies when ${[n]}_3 = 2$ and ${[y]}_3=1$ by Eq.~\ref{Eq2tierTlinksMisRhoSymmetry2}.
This class contains no torus links.

%%%%%%%%%%%%%%%%%%%%%%%%%%%%%%%%%%%%%%%%%%%%%%%%%%%%%%%%%%%%%%%%%%%%%%%%%%%%%%%%%%%%%%%

When ${[n]}_3 = 2 ={[y]}_3$, 
we have $T((3,y),(n,3)) = \widehat{ \sigma_1^2   [1,3]^{y-[y]_3+n-[n]_3 +3}  } $, 
so we may apply (\ref{Jones3brfulltwists1x2y3a}) with $w=2y+2n$ to get
\begin{eqnarray*}
V_{T((3,y),(n,3))} &=& t^{y+n-1} \{  ( 1 + t^{2})  + t^{y+n}  (1+ t^{2}) \}.
\end{eqnarray*}
This class contains no torus links.

Hence, in all cases $T( (3,y ), (n,3) )$ has a representation as 
$\widehat{ \sigma_1^a \sigma_2^b [1,3]^z}$ with $z \equiv 0 \mbox{ mod } 3$.
The detailed description is given in the following proposition.

\begin{prop}
The family of two tier \mbox{T-links} $T( (3,y ), (n,3) )$ are partitioned into six families 
according to their representation by $\sigma_1^a \sigma_2^b [1,3]^z$ with $z \equiv 0 \mbox{ mod } 3$ 
and $z \geq 3$ and $a \geq b \geq 0$:
\begin{enumerate}
\item $a=b=0$ are the torus links $T(3,z) = T( (3, x ), (z-x,3) )$ 
whenever $x \equiv 0 \mbox{ mod } 3$ and $3 \leq x \leq z-6$,
\item $a=b=1$ are the torus links $T(3,z+1)$ which is the same as
   \begin{enumerate}
   \item $T( (3,1+x ), (z-x,3) )$ whenever $x \equiv 0 \mbox{ mod } 3$ and $0 \leq x \leq z-6$,
   \item $T( (3, x), (1+z-x,3) )$ whenever $x \equiv 0 \mbox{ mod } 3$ and $3 \leq x \leq z-3$,
   \end{enumerate}
\item $a=3$ and $b=1$ are the torus links $T(3,z+2)$ which is the same as
   \begin{enumerate}
   \item $T( (3,2+x ), (z-x,3) )$ whenever $x \equiv 0 \mbox{ mod } 3$ and $0 \leq x \leq z-6$,
   \item $T( (3,x ), (2+z-x,3) )$ whenever $x \equiv 0 \mbox{ mod } 3$ and $3 \leq x \leq z-3$,
   \end{enumerate}
\item $a=b=2$ are $T((3,z-2),(4,3) = T((3,1+x),(1+z-x,3))$ 
whenever $x \equiv 0 \mbox{ mod } 3$ and $0 \leq x \leq z-3$,
which includes no torus links,
\item $a=5$ and $b=1$ are $T((3,z-1),(4,3)$, which includes no torus links and is the same as
   \begin{enumerate}
   \item $ T((3,2+x),(1+z-x,3))$
   whenever $x \equiv 0 \mbox{ mod } 3$ and $0 \leq x \leq z-3$,
   \item $ T((3,1+x),(2+z-x,3))$
   whenever $x \equiv 0 \mbox{ mod } 3$ and $0 \leq x \leq z-3$,
   \end{enumerate}
\label{prop2tierTlinkscase51}
\item $a=2$ and $b=0$ are $T((3,z-4),(5,3) = T((3,x-1),(2+z-x,3))$ 
whenever $x \equiv 0 \mbox{ mod } 3$ and $3 \leq x \leq z-3$,
which includes no torus links.
\end{enumerate}
\label{prop2tierTlinks}
\end{prop}

\subsubsection{The Jones polynomial for two tier \mbox{T-links}  with $y=1$, $\rho=2$}
\label{TwoTierYis1Rhois2}
By duality we have $T( (m,1), (n,2)) =  T( (2,n-m), (3,m) )$.
The second form is analyzed in Section~\ref{TwoTierMis2Nis3}.

\subsubsection{The Jones polynomial for two tier \mbox{T-links}  with $m=2$, $n=3$}
\label{TwoTierMis2Nis3}
By Cor.~8 \cite{12}, when $\rho=2$, the braid index is 2; otherwise the braid index is 3, as $\rho \geq 2 = m$.
By duality we have $T( (2,y ), (3,\rho) ) = T( (\rho,1 ), (y+\rho,2) )$.
The Jones polynomial for this family is described by Cor.~\ref{CorJones3brfulltwistsTlinks}.

By Props.~\ref{propJonesUniqueTwistedTorusSum3aplus1} and \ref{propJonesUniqueTwistedTorusSum3aplus2},
there is a representation $\sigma_1^{1+y} \sigma_2 [1,3]^{\rho-1}$ when $\rho \equiv 1 \mbox{ mod } 3$,
and a representation $\sigma_1^{3+y} \sigma_2 [1,3]^{\rho-2}$ when $\rho \equiv 2 \mbox{ mod } 3$.
Thus the only equal links are $T( (2,2+y ), (3,\rho-1) )= T( (2,y ), (3,\rho) )$
whenever $\rho \equiv 2 \mbox{ mod } 3$.

\begin{examp} 
\label{ExhyperbolicTlinksmis2nis3}
The following are hyperbolic Lorenz knots, \cite{12}:
\begin{enumerate}
\item $\mathbf{k}3_{1} = T((2,4),(3,4)) \leftrightarrow \sigma_1^{5} \sigma_2 [1,3]^{3}$,
% T((2,2),(3,4))= 10-124 is NOT hyperbolic - typo in Birman/Kofman paper
% T((2,4),(3,4))= 12n-0242 is hyperbolic and has desired Jones polynomial
% \mathbf{k}3_{1} DOESN'T match Jones polynomial on arXiv for \cite{23}
\item $\mathbf{k}4_{3} = T((2,2),(3,8)) \leftrightarrow \sigma_1^{5} \sigma_2 [1,3]^{6}$,
% \mathbf{k}4_{3} matches Jones polynomial on arXiv for mirror image of \cite{23}
\item $\mathbf{k}5_{5} = T((2,2),(3,11)) \leftrightarrow \sigma_1^{5} \sigma_2 [1,3]^{9}$, 
\item $\mathbf{k}5_{11} = T((2,6),(3,4)) \leftrightarrow \sigma_1^{7} \sigma_2 [1,3]^{3}$,
% \mathbf{k}5_{5} matches Jones polynomial on arXiv
% \mathbf{k}5_{11} matches Jones polynomial on arXiv
\item $\mathbf{k}6_{5} = T((2,2),(3,14)) \leftrightarrow \sigma_1^{5} \sigma_2 [1,3]^{12}$, 
\item $\mathbf{k}6_{19} = T((2,6),(3,5)) \leftrightarrow \sigma_1^{9} \sigma_2 [1,3]^{3}$,
% \mathbf{k}6_{5} matches Jones polynomial on arXiv
% \mathbf{k}6_{19} matches Jones polynomial on arXiv
\item $\mathbf{k}7_{5} = T((2,4),(3,16)) \leftrightarrow \sigma_1^{5} \sigma_2 [1,3]^{15}$, 
\item $\mathbf{k}7_{68} = T((2,6),(3,10)) \leftrightarrow \sigma_1^{7} \sigma_2 [1,3]^{9}$, 
\item $\mathbf{k}7_{90} = T((2,6),(3,8)) \leftrightarrow \sigma_1^{9} \sigma_2 [1,3]^{6}$.
% \mathbf{k}7_{5} matches Jones polynomial on arXiv for mirror image of \cite{23}
% \mathbf{k}7_{68} matches Jones polynomial on arXiv for mirror image of \cite{23}
% \mathbf{k}7_{90} matches Jones polynomial on arXiv for mirror image of \cite{23}
\end{enumerate}

Note that $\mathbf{k}7_{39} = T((3,3),(11,2)) = T((2,8),(5,3)) = T((2,8),(3,5))$ and 
$T((2,8),(3,5)) \leftrightarrow \sigma_1^{11} \sigma_2 [1,3]^{3}$.
Also $\sigma_1^{5} \sigma_2 [1,3]^{3a} = T((3,3a-1),(4,3))$ by Prop.~\ref{prop2tierTlinks} case~\ref{prop2tierTlinkscase51}.
\end{examp}

%%%%%%%%%%%%%%%%%%%%%%%%%%%%%%%%%  3-tier T-links with braid index 3 %%%%%%%%%%%%%%%%%
%%%%%%%%%%%%%%%%%%%%%%%%%%%%%%%%%  3-tier T-links with braid index 3 %%%%%%%%%%%%%%%%%

\subsection{The Jones polynomial for 3-tier T-links with braid index 3}
\label{SS3tier}

The 3-tier T-links have the form $T((r_1, s_1), (r_2, s_2), (r_3, s_3))$, 
with dual form $T((s_3, r_3-r_2), (s_3+s_2, r_2-r_1), (s_3+s_2+s_1, r_1))$, \cite{12}.

First consider the case when $r_1 = s_3$. Using Cor.~8, \cite{12}, we have
$i_0=\min\{i: r_i \geq \overline{r_{3-i}} \} = 2$, and 
$j_0=\min\{j: \overline{r_j} \geq r_{3-j} \} = 2$.
To obtain a braid index of $3$ we need $3=\min\{r_2, \overline{r_{2}} \}$.
This implies $r_1=2$, with either $r_2=3$ or $s_2=1$.
In the former case we have $T((2, s_1), (3, s_2), (r_3, 2))$ with dual form
$T((2, r_3-3), (2+s_2, 1), (2+s_2+s_1, 2))$.
In the latter case we have $T((2, s_1), (r_2, 1), (r_3, 2))$, 
with dual form $T((2, r_3-r_2), (3, r_2-2), (3+s_1, 2))$.

As these four forms are readily seen to be interchangeable, we study the form 
$T((2, s_1), (3, s_2), (r_3, 2)) = \widehat{[1, 2]^{s_1} [1, 3]^{s_2} [2,1]^{r_3-3} [1, 3]^2}$,
by Prop.~\ref{PropBraidRel}. This is just
\begin{enumerate}
\item $ % \widehat{\sigma_1^{s_1+r_3-3} [1, 3]^{s_2+2} } =
\widehat{\sigma_1^{s_1+r_3} \sigma_2 [1, 3]^{s_2} }$, i.e. $T((2,s_1+r_3-1),(3,s_2+1))$
when $s_2 \equiv 0 \mbox{ mod } 3$,
\item $ %\widehat{\sigma_1^{s_1} (\sigma_1 \sigma_2) \sigma_1^{r_3-3} [1, 3]^{s_2+1}}= 
\widehat{\sigma_1^{s_1} \sigma_2^{r_3-3} [1, 3]^{s_2+2}}$,
when $s_2 \equiv 1 \mbox{ mod } 3$,
\label{Vfor3tiersymmetrics2is1mod3}
   \begin{enumerate}
   \item this is the \mbox{T-link} $T((2,s_1-1),(3,s_2+3))$ when $r_3=4$,
   \item this is the \mbox{T-link} $T((2,r_3-4),(3,s_2+3))$ when $s_1=1$,
   \end{enumerate}
\item $ %\widehat{\sigma_1^{s_1} (\sigma_1 \sigma_2)^2 \sigma_1^{r_3-3} [1, 3]^{s_2}}= 
\widehat{\sigma_1^{s_1+1} \sigma_2^{r_3-2} [1, 3]^{s_2+1}}$
when $s_2 \equiv 2 \mbox{ mod } 3$.
\label{Vfor3tiersymmetrics2is2mod3}
\end{enumerate}
Hence $T((2, s_1), (3, s_2), (r_3, 2))$ always has a representation as $\widehat{ \sigma_1^{x} \sigma_2^y [1,3]^{z} }$
with $z \equiv 0 \mbox{ mod } 3$. 
The Jones polynomial for case~\ref{Vfor3tiersymmetrics2is1mod3} when $s_1>1$ and $r_3>4$ 
is described by (\ref{Jones3brfulltwists1x2y3aA}) and its
coefficient vector is described in Section~\ref{SSCoeffVectorszis0mod3}; 
this also applies to case~\ref{Vfor3tiersymmetrics2is2mod3}.
The following result summarizes all the non-trivial forms that represent the same link.

\begin{prop}
The following is an exhaustive list of the 3-tier \mbox{T-links} with $r_1 = s_3$ and
braid index 3 which have no 2-tier representation. 
These have a braid representation $\sigma_1^{x} \sigma_2^y [1, 3]^{z} $, or
$\sigma_1^{y} \sigma_2^x [1, 3]^{z} $, 
where $x,y \geq 2$ and $(x,y) \neq (2,2)$ and $z \equiv 0 \mbox{ mod } 3$ with $z \geq 3$.
\begin{enumerate}
\item $T((2, x), (3, z-2), (3+y, 2))$ or $T((2, y), (3, z-2), (3+x, 2))$ from \ref{Vfor3tiersymmetrics2is1mod3},
\item dual to prior item: $T((2, y), (z, 1), (x+z, 2))$ or $T((2, x), (z, 1), (y+z, 2))$,
\item $T((2, x-1), (3, z-1), (2+y, 2))$ or $T((2, y-1), (3, z-1), (2+x, 2))$ from \ref{Vfor3tiersymmetrics2is2mod3}
\item dual: $T((2, y-1), (z+1,1), (x+z, 2))$ or $T((2, x-1), (z+1,1), (y+z, 2))$.
\end{enumerate}
\label{prop3tierSymdups}
\end{prop}
Note that if we map $T((2, r_3-r_2), (3, r_2-2), (3+s_1, 2))$ to $\sigma_1^{x} \sigma_2^y [1, 3]^{z} $,
we obtain $T((2, y), (z, 1), (x+z, 2))$, which appears in the list above, as does the dual form.
The form for mapping $T((2, r_3-r_2), (3, r_2-2), (3+s_1, 2))$ to $\sigma_1^{y} \sigma_2^x [1, 3]^{z} $,
and its dual, also appear in the list above.

\begin{examp}
$\mathbf{k}7_{99} =T((2,2),(3,2),(5,2))$ is a hyperbolic Lorenz knot, \cite{12}.
By case~\ref{Vfor3tiersymmetrics2is2mod3} preceding Prop.~\ref{prop3tierSymdups}
it is generated by $\sigma_1^{3} \sigma_2^3 [1, 3]^{3} $.
\end{examp}

\medskip

Now consider the case when $r_1 > s_3$, hence $r_1 \geq 3$ and $i_0 = 1$ or $2$.
Prop.~\ref{PropBraidIndexTlink} says that in case $i_0=1$ and $r_1> \overline{r_2}$
we have $j_0=3$. In order for the braid index to be three, we must have $3=r_1$, as 
$\overline{r_3} \geq 4$. Thus $s_3=2$ and $\overline{r_2} =s_3+s_2 \geq 3$ which
contradicts the assumption $r_1> \overline{r_2}$.
If $i_0=2$, we must have $j_0=2$ as $j_0=1$ implies $\overline{r_1}=s_3 \geq r_2$.
In order for the braid index to be three, we must have $\overline{r_2}=3$ and $\overline{r_1}=s_3 =2$.
This contradicts the assumption $i_0=2$, as $3 \leq r_1 = \overline{r_2}$.

The prior paragraph shows that the only viable case is $i_0=1$ and $r_1= \overline{r_2}$,
so that $j_0=2$. In order for the braid index to be three, we must have $s_3=2$, $s_2=1$ and $r_1 = 3$.
This gives rise to the \mbox{T-link} $T((3, s_1), (r_2, 1), (r_3, 2))$, 
with dual form $T((2, r_3-r_2), (3, r_2-3), (3+s_1, 3))$.
Prop.~\ref{PropBraidRel} shows the latter may be represented by the braid word 
$[1,2]^{r_3-r_2} [1,3]^{r_2-3} [3,1]^{s_1} [1,3]^{3} $.
We have the following representations:
\begin{enumerate}
\item $[1,2]^{r_3-r_2} [1,3]^{s_1+r_2}$, i.e. $T((2,r_3-r_2),(3,s_1+r_2))$ when $s_1 \equiv 0 \mbox{ mod } 3$,
\item $[1,2]^{r_3-r_2+2}[1,3]^{s_1+r_2-1}$, i.e. $T((2,r_3-r_2+2),(3,s_1+r_2-1))$ when $s_1 \equiv 2 \mbox{ mod } 3$,
\item when $s_1 \equiv 1 \mbox{ mod } 3$,
   \begin{enumerate}
   \item $\sigma_1^{r_3-r_2+1} \sigma_2 [1,3]^{s_1+r_2-1}$, i.e. $T((2,r_3-r_2),(3,s_1+r_2))$ when $r_2 \equiv 0 \mbox{ mod } 3$,
   \item $\sigma_1^{r_3-r_2+2} \sigma_2^2 [1,3]^{s_1+r_2-2}$ when $r_2 \equiv 1 \mbox{ mod } 3$,
   \label{Vfor3tierAsymmetrics1is1mod3r2is1mod3}
   \item $\sigma_1^{r_3-r_2+5} \sigma_2 [1,3]^{s_1+r_2-3}$, i.e. $T((2,r_3-r_2+4),(3,s_1+r_2-2))$ when $r_2 \equiv 2 \mbox{ mod } 3$.
   \end{enumerate}
\end{enumerate}

To see the result when $s_1 \equiv 2 \mbox{ mod } 3$, first note that $[1,2]^{r_3-r_2} [1,3]^{s_1+r_2-2} [3,1]^{2}$ is a 
braid representative.
For the case $s_1 \equiv 1 \mbox{ mod } 3$ and $r_2 \equiv 2 \mbox{ mod } 3$, note that 
$[1,2]^{r_3-r_2} (\sigma_1\sigma_2)^2(\sigma_2\sigma_1) [1,3]^{s_1+r_2-3}$ is a representative, 
which is equivalent to $[1,2]^{r_3-r_2+2}    \sigma_1\sigma_2 \sigma_1^2 [1,3]^{s_1+r_2-3}$.

The case when $r_1 < s_3$ gives rise to a dual form with $\overline{r_1} < \overline{s_3}$,
in which case the prior analysis applies to the dual form.
Thus the 3-tier T-links of the form $T((r_1, s_1), (r_2, s_2), (r_3, s_3))$ with braid index three 
always have a representation as $\widehat{ \sigma_1^{x} \sigma_2^y [1,3]^{z} }$
with $z \equiv 0 \mbox{ mod } 3$.
The Jones polynomial for case~\ref{Vfor3tierAsymmetrics1is1mod3r2is1mod3}
is described by (\ref{Jones3brfulltwists1x2y3a}) and its
coefficient vector is described in Section~\ref{SSCoeffVectorszis0mod3}.
The following result summarizes all the true 3-tier forms, item
\ref{Vfor3tierAsymmetrics1is1mod3r2is1mod3}, that represent the same link.

\begin{prop}
The following is an exhaustive list of the 3-tier \mbox{T-links} with $r_1 \neq s_3$ and
braid index 3 which have no 2-tier representation.
These have a braid representation $\sigma_1^{x} \sigma_2^y [1, 3]^{z} $, or
$\sigma_1^{y} \sigma_2^x [1, 3]^{z} $, 
where $x,y \geq 2$ and $z \equiv 0 \mbox{ mod } 3$ with $z \geq 3$.
In fact, $\min(x,y)=2$, and $3 \leq \max(x,y)=M$. 
For any choice of $r_2 \equiv 1 \mbox{ mod } 3$ with $4 \leq r_2 \leq z+1$ we have:

\begin{enumerate}
\item $T((2, M-2), (3, r_2-3), (z+5-r_2, 3))$, or
\item dual to prior item: $T((3, z+2-r_2), (r_2, 1), (M+r_2-2, 2))$.
\end{enumerate}
\label{prop3tierAsymdups}

Note that these are a subset of the links described by Prop.~\ref{prop3tierSymdups}.
\end{prop}

%%%%%%%%%%%%%%%%%%%%%%%%%%%%%%%%%  4-tier T-links with braid index 3 %%%%%%%%%%%%%%%%%
%%%%%%%%%%%%%%%%%%%%%%%%%%%%%%%%%  4-tier T-links with braid index 3 %%%%%%%%%%%%%%%%%

\subsection{The Jones polynomial for 4-tier T-links with braid index three}
Cor.~\ref{cormaxtiers} shows that $T((2, s_1), (3, s_2), (r_3, 1),  (r_4, 2))$
is the general form for 4-tier \mbox{T-links} with braid index 3. 
%with dual form $T((2, r_4-r_3), (3, r_3-3), (s_4+ s_3+s_2, 1), (s_4+s_3+s_2+s_1, 2))$, \cite{12}.
Proposition~\ref{PropBraidRel} tells us this has a representation as 
%$[1,2]^{s_1} [1,3]^{s_2} [1,r_3] [1,r_4]^2$.
$[1,2]^{s_1} [1,3]^{s_2} [1,r_3] [2,1]^{r_4-r_3} [1,r_3]^2=$ 
$[1,2]^{s_1} [1,3]^{s_2} \sigma_2^{r_4-r_3} [1,r_3]^3$.
Invoke Prop.~\ref{PropBraidRel} again to obtain $[1,2]^{s_1} [1,3]^{s_2} \sigma_2^{r_4-r_3} [3,1]^{r_3-3} [1,3]^3$,
and this may be rewritten as 
$\sigma_1^{s_1} [1,3]^{[s_2]_3} \sigma_2^{r_4-r_3} [3,1]^{[r_3]_3} [1,3]^{s_2-[s_2]_3+r_3-[r_3]_3}$.

We have the following representations when $s_2 \equiv 0 \mbox{ mod } 3$:
\begin{enumerate}
\item $\sigma_1^{s_1} \sigma_2^{r_4-r_3} [1,3]^{s_2 +r_3}$,
when $r_3 \equiv 0 \mbox{ mod } 3$,
   \begin{enumerate}
   \item this is $T((2,r_4-r_3-1),(3,s_2 +r_3+1))$ when $s_1=1$,
   \item this is $T((2,s_1-1),(3,s_2 +r_3+1))$ when $r_4-r_3=1$,
   \end{enumerate}
\item $\sigma_1^{s_1+1} \sigma_2^{r_4-r_3+1} [1,3]^{s_2 +r_3-1}$,
when $r_3 \equiv 1 \mbox{ mod } 3$,
\item $\sigma_1^{s_1 +r_4-r_3+3}  \sigma_2  [1,3]^{s_2 +r_3-2}$, i.e. $T((2,s_1 +r_4-r_3+2), (3,s_2 +r_3-1))$
when $r_3 \equiv 2 \mbox{ mod } 3$.
\end{enumerate}
To see the result when $r_3 \equiv 2 \mbox{ mod } 3$, note that $\sigma_1^{s_1} \sigma_2^{r_4-r_3} [3,1]^{2} [1,3]^{s_2 +r_3-2}$
is just %$\sigma_1^{s_1} \sigma_2^{r_4-r_3} \sigma_1 \sigma_2\sigma_1 \sigma_1 [1,3]^{s_2 +r_3-2}$.
$\sigma_1^{s_1} (\sigma_1 \sigma_2 \sigma_1^{r_4-r_3})  \sigma_1 \sigma_1 [1,3]^{s_2 +r_3-2}$.

\medskip

We have the following representations when $s_2 \equiv 1 \mbox{ mod } 3$:
\begin{enumerate}
\item $\sigma_1^{s_1+1} \sigma_2^{r_4-r_3+1} [1,3]^{s_2-1+r_3}$,
when $r_3 \equiv 0 \mbox{ mod } 3$,
\item $\sigma_1^{s_1+2}   \sigma_2^{r_4-r_3+2} [1,3]^{s_2-1+r_3-1}$,
when $r_3 \equiv 1 \mbox{ mod } 3$,
\item $\sigma_1^{s_1+r_4-r_3+5} \sigma_2 [1,3]^{s_2-1+r_3-2}$, i.e. $T((2,s_1 +r_4-r_3+4), (3,s_2 +r_3-2))$
when $r_3 \equiv 2 \mbox{ mod } 3$.
\end{enumerate}

We have the following representations when $s_2 \equiv 2 \mbox{ mod } 3$:
\begin{enumerate}
\item $\sigma_1^{s_1+r_4-r_3+3} \sigma_2 [1,3]^{s_2-2+r_3}$, $T((2,s_1 +r_4-r_3+2), (3,s_2 +r_3-1))$
when $r_3 \equiv 0 \mbox{ mod } 3$,
\item $\sigma_1^{s_1+r_4-r_3+5} \sigma_2 [1,3]^{s_2-2+r_3-1}$, i.e. $T((2,s_1 +r_4-r_3+4), (3,s_2 +r_3-2))$
when $r_3 \equiv 1 \mbox{ mod } 3$,
\item $\sigma_1^{s_1 +r_4-r_3+2} [1,3]^{3+s_2-2+r_3-2}$, i.e. $T((2,s_1 +r_4-r_3+2),(3,s_2+r_3-1))$
when $r_3 \equiv 2 \mbox{ mod } 3$.
\end{enumerate}

The Jones polynomial for cases which are not identified as two-tier \mbox{T-links} are
described by (\ref{Jones3brfulltwists1x2y3a}) and the
coefficient vector is described in Section~\ref{SSCoeffVectorszis0mod3}.
The prior calculations are summarized next.
Recall that $\overline{s_{1}} = r_{k}-r_{k-1}$.

\begin{prop}
A 4-tier \mbox{T-link} with braid index 3 has a 2-tier or 3-tier form.

%If $L_1$ and $L_2$ are 4-tier \mbox{T-links} with braid index 3 with
Suppose $L_i = T((2, s_{1,i} ), (3, s_{2,i} ), (r_{3,i}, 1),  (r_{4,i} , 2))$ for $i=1,2$ and
$\{ s_{1,1}, \overline{s_{1,1}} \} = \{ s_{1,2}, \overline{s_{1,2}} \}$
and $\{ [s_{2,1}]_3, [r_{3,1}]_3 \} = \{ [s_{2,2}]_3, [r_{3,2}]_3 \}$
and $ s_{2,1} + r_{3,1} = s_{2,2} + r_{3,2} $.

It follows that $L_1=L_2$.
\label{prop4tierdups}
\end{prop}

\subsection{Classification result for three-braid T-links}
The results in the prior sections may be combined to give us the following result,
which extends the classification result for knots in $\mathcal{L}_3$
by R.~Bedient, \cite{4}.
\begin{thm}
Suppose $\beta(x,y,z)=\sigma_1^x \sigma_2^y [1,3]^z$ with $z \equiv 0 \mbox{ mod } 3$, and $z \geq 3$ and $x\geq y \geq 0$.

If $L$ is a \mbox{T-link} of braid index three, then it has a unique representation as some $\beta(x,y,z)$ and
$L$ has a \mbox{T-link} representation with three or fewer tiers:
\begin{enumerate}
\item when $x=y=0$, we have $L=T(3,z)$,
\item when $x>y=0$, we have $L=T((2,x),(3,z))$, with no torus links present,
\item when $x=y=1$, we have $L=T(3,z+1)$,
\item when $x=3$ and $y=1$, we have $L=T(3,z+2)$,
\item when $x >y=1$, we have $L=T((2,x-1),(3,z+1))$, with no torus links present for $x \neq 3$,
\item when $x=y=2$, we have $L = T((3,z-2),(4,3))$, which is not a torus link,
\item when $x,y \geq 2$, we have $L=T((2,x-1),(3,z-1),(2+y,2))$, with no torus links present. 
The only 2-tier \mbox{T-link} present is when $x=y=2$.
\end{enumerate}
The closure of any such braid word, $\beta(x,y,z)$, is a \mbox{T-link} of braid index three.
\label{thmTlinkbr3}
\end{thm}

Note that we cannot allow $z=0$, since such a braid word represents the connected sum of two 
elementary torus links, $T_x \sharp T_y$, while Cor.~1, \cite{12}, indicates that \mbox{T-links} are prime,
and of course in this setting $x \neq 1$ and $y \neq 1$.

It remains open to find what characteristics of a general \mbox{T-link} lead to a Jones polynomial 
with coefficients which are both ''sparse'' and ''low'' in value. ''Sparse'' would refer to 
relatively few non-zero values within each block, as \cite{22} shows that the blocks separate as 
the number of full twists increases.
It would be interesting to
know when the blocks in the coefficient vector are disjoint, what is their length, 
when are the inter-block coefficients non-zero when full twists on an even number of strands are added,
and what is their value, even if the full Jones polynomial could not be calculated.

The family of Lorenz links of braid index three has the interesting property that
the Jones polynomial distinguishes distinct links within this family.
This is not the general behavior across Lorenz links as there are 
distinct Lorenz knots that have the same Jones polynomial (see \cite{12}).
It was shown in Prop.~2.25, \cite{25}, that the Jones polynomial distinguishes distinct
connected sums of elementary torus links.

\begin{rem}
It would be interesting to know whether the Jones polynomial distinguishes distinct links within the family
$\{ \widehat{ [1,n]^z \prod_{i=1}^{n-1} \sigma_i^{x_i}  }  : 
z \equiv 0 \mbox{ mod } n,\, \mbox{and } z,\,x_i \geq 0 \mbox{ for all } i\}$, as this 
is true for $z=0$, or $n=2, 3$.
\end{rem}

\begin{rem}
Cor.~\ref{cormaxtiers} shows that \mbox{T-links} with braid index $b$ must have $2b-2$ or fewer tiers.
However for both $b=2, 3$ the \mbox{T-link} with the maximum number of tiers, 
$T((2,s_1), (3,s_2),\ldots, (b,s_{b-1}), (r_{b},1),\ldots, (r_{2 b-3},1),(r_{2 b-2},2))$,
has a representation with fewer tiers.
It would be interesting to establish whether this is always true. A related question is 
whether there is always some link that has a representation with $2b-3$ tiers but no fewer tiers suffice
as this is true for $b \leq 3$.
\end{rem}

\section*{Acknowledgments}
%This section should come before the References. Funding 
%information may also be included here.
The author would like to thank J.~Birman for her suggestion to investigate the
Jones polynomial of Lorenz links, which leads to many other interesting questions.
I.~Kofman made the helpful suggestion to look at R.~Bedient's paper, \cite{4}.

\end{document}